\documentclass{article}
\usepackage[english]{babel}
\usepackage[margin=3.3cm]{geometry}
\usepackage{url}
\usepackage{titlesec}
\usepackage{ulem}
\titlespacing\section{0pt}{0pt plus 0pt minus 2pt}{0pt plus 0pt minus 0pt}
\titlespacing\subsection{0pt}{0pt plus 0pt minus 2pt}{0pt plus 0pt minus 0pt}
\titlespacing\subsubsection{0pt}{0pt plus 2pt minus 0pt}{0pt plus 0pt minus 0pt}
\usepackage{appendix,etoolbox}
\usepackage{authblk} 
\usepackage{float}
\usepackage{amsmath,amsthm,bbm,amssymb, amsfonts}   
\usepackage{thmtools}
\usepackage{indentfirst} 
\usepackage{csquotes}
\usepackage{xcolor}
\usepackage{multirow}
\usepackage{comment}
\usepackage{enumerate}
\usepackage{cancel}
\usepackage{stmaryrd}
\usepackage[titles]{tocloft}
\usepackage{setspace}
\usepackage{enumitem}
\usepackage{ifthen}
\usepackage{hyperref}
\usepackage{cite}
\usepackage{enumerate}
\newtheorem{theorem}{Theorem}[section]
\newtheorem{Prop}[theorem]{Proposition}

\newtheorem{lemma}[theorem]{Lemma}
\newtheorem*{lemma*}{Lemma}
\newtheorem{Lemma}[theorem]{Lemma}
\newtheorem{definition}[theorem]{Definition}

\newtheorem{assumption}[theorem]{Assumption}

\newtheorem*{assumption*}{Assumption}
\newtheorem*{assumptions*}{Assumptions}
\newtheorem*{definition*}{Definition}

\newcommand{\N}{\mathbb{N}}

\newcommand{\R}{\mathbb{R}}

\newcommand{\E}{\mathbb{E}}

\newcommand{\mtcX}{\mathcal{X}}

\newcommand{\mtcB}{\mathcal{B}}

\newcommand{\mtcU}{\mathcal{U}}
\newcommand{\mtcF}{\mathcal{F}}

\def\cmb#1{\marginpar{\raggedright\tiny{\textcolor{red}{Bertrand} : \textcolor{purple}{#1}}}}
\def\cmv#1{\marginpar{\raggedright\tiny{\textcolor{red}{Vincent} : \textcolor{blue}{#1}}}}
\def\cmt#1{\marginpar{\raggedright\tiny{\textcolor{red}{Tresnia} : \textcolor{magenta}{#1}}}}
\begin{document}

\title{Strong law of large numbers and $L\log L$ condition for supercritical  branching processes}
\author{Vincent Bansaye\footnote{CMAP, INRIA, École polytechnique, Institut Polytechnique de Paris, 91120 Palaiseau, France}, \, Tresnia Berah\footnote{Department of Mathematics - Statistics Section, Imperial College, London, SW7 2AZ, United-Kingdom},  \, Bertrand Cloez \footnote{MISTEA, Université de Montpellier, INRAE, Institut Agro, Montpellier, France}}
\date{\today}

\maketitle

\begin{abstract}
We consider branching processes for structured populations: each individual is characterised by a type or trait which belongs to a general measurable state space. We focus on the supercritical recurrent case, where the population may survive and grow and the trait distribution converges to a probability measure. The branching process is then expected to be driven by the positive triplet of first eigenvalue problem of the first moment semigroup. Under the assumption of convergence of the renormalized semigroup in weighted total variation norm, we prove strong convergence of the renormalized empirical measure and non-degeneracy of  the limit of the martingale. Convergence is obtained under an $L\log L $ condition which provides a  Kesten-Stigum result in infinite dimension and relaxes the uniform convergence  of the first moment semigroup  in the work of Asmussen and Hering in 1976. The techniques of proof combine families of martingales and contraction properties and  the truncation procedure of Asmussen and Hering. These results  unify part of the literature and  capture new situations, as illustrated by   absorbed branching diffusion, the house of cards  model and some growth-fragmentation processes.
\end{abstract}

\noindent\textit{\small{Keywords: branching processes, martingales, weighted total variation norm, contraction of semi-groups, Kesten-Stigum condition, Lyapunov functions}}


\section{Introduction and main results}\label{sec:intro-results}
For a supercritical Galton-Watson process $Z$ with mean number of offspring $m>1$, the necessary and sufficient condition for the non-degeneracy of the limit $W$ of the martingale $W_n=Z_n/m^n$ is the famous $L\log L$ 
moment condition on the reproduction law $L$. The limit of the martingale is then finite and positive on the survival event \cite{kesten1966limit,lyons1995conceptual}.
The generalisation of the asymptotic behaviour to  the multitype case, with finite type set  $\mathcal X$, is also  known  from the  work of Kesten and Stigum \cite{kesten1966limit, kurtz1997conceptual}.

More precisely, when $\mathcal X$ is finite and the mean matrix $S=(S_{x,y})_{x,y\in \mathcal X}$ of the reproduction law is primitive (i.e.~irreducible and aperiodic), the Perron Frobenius theorem can be invoked. It  ensures that there exists a unique triplet $(\lambda, h,\gamma)$
where $h$ is a  positive function on $\mathcal X$, $\gamma$ is a probability  on $\mathcal X$  and  $\lambda \in (0,\infty)$
such that

$$Sh=\lambda h, \quad \gamma S= \lambda \gamma,  \quad\gamma(h)=\sum_{x\in \mathcal X} h(x)\gamma_x=1.$$ 
Moreover
$$ \, S^n_{x,y} \sim_{n\rightarrow \infty} \lambda^n \,  h(x) \, \gamma_y,$$ for any $x,y\in \mathcal X$.
The underlying convergence is exponential and uniform on $\mathcal X$ since this latter is finite. 
In the multitype setting, the branching process $Z_n=(Z_n^i : i \in \mathcal X)$ counts the number of individuals of each type $i$ in generation $n$. 
Equivalently, $Z$ can be represented by its empirical measure and at an individual level.
Denoting by $\mathbb G_n$  the individuals of generation $n$, the branching process can be defined by  
$$Z_n=\sum_{u\in \mathbb G_n} \delta_{Z(u)},$$
where $Z(u)$ is the type of individual $u$. We refer to Section \ref{sec:construction} for details. 
Thus, for any $i\in \mtcX$, $Z_n^i=Z_n(\{i\})$, and for $f$ non-negative function on $\mathcal X$,

$$ Z_n(f)=\sum_{u\in \mathbb G_n} f(Z(u))=\sum_{i\in \mathcal X} f(i) Z_n^i.$$
In  this  paper, we adopt  the semigroup and linear operator framework, which will be relevant in particular when $\mathcal X$ will be infinite. We consider the first moment semigroup of the multitype branching process defined for $x\in \mathcal X$ and $f$ non-negative function on $\mathcal X$ by 
$$S_nf(x)=\E_{\delta_x}(Z_n(f))=\E\left(\sum_{i\in \mathcal X} f(i) Z_n^i \, \vert \,Z_0=\delta_x\right).$$ We observe that $S_{x,y}=S_1 \mathbf{1}_{y}(x)=\E(Z_1^y \, \vert \, Z_0=\delta_x).$
Uniform convergence as per Perron Frobenius result can  be written as 
\begin{align}\label{UCsg}
\sup_{x\in \mathcal X, \vert f\vert \leq h} \bigg\vert\frac{S_nf(x)}{ \lambda^nh(x)\gamma (f)}-1\bigg\vert
\stackrel{n\rightarrow \infty}{\longrightarrow} 0.
\end{align}
Several versions  of the Kesten-Stigum theorem have been obtained in infinite type spaces, starting from works in the countable space under  some $L^2$ moment conditions \cite{Moy1967} and uniform bounds on the mean matrix \cite{harris1963theory}. Up to our knowledge,  the single  statement extending the finite dimensional case  to general  (measurable) state space $\mathcal X$ was obtained by \cite{asmussen1976strong}. More precisely, starting from one single individual with type $x\in \mathcal X$, they show that there exists a non-negative random variable $W$ such that 
$$\E_{\delta_x}(W)=h(x); \qquad \lim_{n\rightarrow \infty} \frac{Z_n(f)}{\lambda^n}=\gamma(f)\, W  \quad \mathbb P_{\delta_x} \, \quad \text{ a.s.}$$
Their result is proved assuming uniform convergence as per \eqref{UCsg} and the following moment condition  
\begin{align}
\label{LlogLh}
\E_{\gamma}(Z_1(h)\log^{\star} Z_1(h))<\infty,
\end{align}
where the continuous function $x\mapsto \log^{\star}x$ on $[0,\infty)$ is defined as
\begin{align*}
    \log^{\star}x =\begin{cases}
       x/e, \quad 0\leq x \leq e\\
        \log x, \quad x>e.
    \end{cases}
\end{align*}
On one hand, the previous $L\log L$ condition  \eqref{LlogLh} is optimum. But on the other hand, the uniform convergence \eqref{UCsg} may not be easily satisfied. And even when satisfied, it may actually be difficult to prove. We mention here recent works by \cite{horton2020stochastic, gonzalez2022asymptotic} and references therein,  with applications to neutron transport. 
Actually, Condition \eqref{UCsg} is reminiscent of the finite dimensional case. In particular, it  forces the second eigenfunction to be dominated by the first one.
This uniform convergence is not satisfied in general for branching processes with non-bounded type space, or with bounded space but whose behaviour on {the boundary} is degenerate. For instance, if we consider neutral models whereby 
the mean number of offspring does not depend on the type $x\in \mathcal X$, the harmonic function is $h=1,$
and Condition \eqref{UCsg} in this case would imply that the Markov chain following the typical type  along the spine comes down from infinity \cite{bansaye2011limit}. This, in turn, means that the Markov chain comes back to compact sets very fast when it starts from large values and therefore excludes behaviours like random walk with negative drift or subcritical branching process. Furthermore, the uniformity of convergence with respect to $\gamma(f)$  in \cite{asmussen1976strong}--\eqref{UCsg}  excludes interesting cases where the limiting distribution $\gamma$ admits  a density but the branching process has no diffusive components.
More generally, the study respectively of branching Brownian motion with absorption, of branching diffusion and of growth fragmentation, have all motivated the relaxation of the uniform convergence of the renormalized first moment, see e.g.~\cite{englander2009law,louidor2020strong, englander2010strong,  bertoin2020strong, tomavsevic2022ergodic, horton2020strong, bansaye2023growth,cloez2017limit,zbMATH07199305,chen2017law}. These aforementioned works obtain strong laws of large numbers for certain classes of branching processes. 
They require in general $L^p$ moment condition with $p>1$, but also a certain knowledge of spectral elements and quantitative estimates of the first moment semigroup.\\ 
 
Our aim in this paper is to present a general Kesten-Stigum theorem which help to unify the literature. We also want to cover new cases and models in continuous-time  and space. In practice, we   relax uniformity 
in \eqref{UCsg} whilst keeping minimal moment conditions and proving strong convergence of the empirical measure. 
The proof relies on  a family of martingales originating from martingale increments, together
with contraction properties of the first moment semigroup for weighted total variation norm. It also involves 
martingale decomposition into  an $L^1$ and an $L^2$ part coming from
 the subtle truncation procedure  of \cite{asmussen1976strong}. \\
Let us explain and motivate more the  contraction which plays a key role in our estimates and   leads to our assumption on the ergodic profile of the first moment semigroup. 
The existence of a stationary regime for the type distribution is directly linked to the ergodicity of the Markov chain of a typical individual and the existence of a Lyapunov function $V$ for the associated Markov kernel. We refer to  \cite{bansaye2011limit} for the neutral case, where it is simpler, and more broadly to \cite {Nummelin, cloez2017limit}. More generally, different  approaches have allowed to quantify the convergence of non conservative semigroups,  and in particular to obtain geometric  convergence in (possibly weighted) total variation norm  \cite{kontoyiannis2003spectral, kontoyiannis2012geometric, del2004feynman, del2002stability, del2023stability, bansaye2022non,velleret2023exponential, champagnat2023general, champagnat2016exponential}.
These works provide sufficient (and sometimes necessary) conditions   on the semigroup $(S_n)_n$ so that
 for any $x\in \mathcal X$ and any $n\in \mathbb N$,
\begin{align}
\label{ergodncc}
\sup_{\vert f \vert \leq V} \bigg\vert\frac{S_nf(x)}{ \lambda^n}-h(x)\gamma(f)\bigg\vert \leq C V(x) \eta ^n,
\end{align}
where $\eta \in (0,1)$ and $V\geq h$ is a Lyapunov function which quantifies the impact of the initial type in the speed of convergence and $C$ is a constant. The condition \eqref{UCsg} of \cite{asmussen1976strong}  amounts to requiring $V=h$. Among the new situations covered by our results, the case where $V=1$ and $h$ vanishes at the boundary  is interesting, see Section \ref{sec:Bdiffusion} for an example. The case of a non-bounded domain $\mathcal X,$ where $V$ grows to infinity faster than $h$ provides another relevant class of examples, see for instance \cite{bansaye2022non, tomavsevic2022ergodic, bansaye2023growth} and our last application in Section \ref{GFMod}. 
 We also want to allow sub-geometric convergence of the renormalized semigroup, as in \cite{asmussen1976strong}. We are motivated in particular by polynomial speed and refer to \cite{canizo2023harris, cloez2024fast} and  to Section \ref{sec:HouseOfCards} for details.

Let us now turn to a more formal presentation of the  main result. We will detail the general construction of our branching process and its semigroup in the next section. In this work, we consider a measurable space $\mathcal X$ with an underlying topological space equipped with its Borel $\sigma$-field $\mtcB_{\mtcX}$. As explained above, our main assumption concerns the first moment semigroup $(S_n)_n$ associated to the branching process.
\begin{assumption}
\label{ass:cvSt}
There exists a measurable function $V^\star : \mathcal X\rightarrow (0,\infty)$  and a triplet $(\lambda,\gamma,h)$
such that $\lambda >1$, $\gamma$ is a probability on $\mathcal X$, $h : \mathcal X\rightarrow (0,\infty)$ is measurable and  
$$ \gamma S_1 =\lambda \gamma, \quad S_1 h=\lambda h, \quad  \gamma(h)=1, \quad \sup_{\mathcal X} h/ V^\star<\infty, \quad \gamma(V^\star)<\infty.$$
Besides, there exists a sequence of non-negative real numbers $(a_n)_{n\in \N}$  such that
for any $n\geq 0$ and $x\in \mathcal X$,
\begin{align} \label{ergodnc}
\sup_{\vert f \vert \leq V^\star} \bigg\vert\frac{S_nf(x)}{ \lambda^n}-h(x)\gamma(f)\bigg\vert \leq V^\star (x) a_n, \qquad \sum_{k\geq 1} \frac{a_k}{k}<\infty.
\end{align}
\end{assumption}
We observe here that \cite{asmussen1976strong} only needs to assume that $a_n$ tends to $0,$ whereas we need slightly more.  On the other hand, we relax the uniform convergence by allowing $V^\star$ to be large compared to $h$.

\begin{theorem}\label{th:main-discret}
Under {Assumption \ref{ass:cvSt}}, we further suppose that there exists a measurable function $V: \mathcal X \to (0,\infty)$ such that  $\sup_{\mathcal{X}}h/V<\infty$, $V\leq V^\star$ and for any $k \geq 0$,
 \begin{align}\label{lloglass}
\sup_{x\in\mathcal{X}}\frac{\E_{\delta_x}(Z_k(V)\log^\star Z_k(V))}{V^\star (x)}<\infty.   \end{align}
Then, for any  $x\in \mathcal X$, 
\begin{align*}
 \lim_{n\rightarrow \infty}    \frac{Z_n(h)}{\lambda^n}= \, W, \quad  \mathbb P_{\delta_x} \text{ a.s. and in } \ L^1, \quad \E_{\delta_x}(W)=h(x).
\end{align*} Besides, for any $f$ such that $\sup_{\mathcal{X}}f/V<\infty$,
\begin{align*}
\lim_{n\rightarrow \infty}    \frac{Z_n(f)}{\lambda^n}=\gamma(f) \, W,\quad  \mathbb P_{\delta_x} \text{ a.s. and in } \ L^1.
\end{align*}
\end{theorem}
As explained at the beginning of this section, and as the reader will see in the last sections, several contraction results and spectral techniques  allow to prove the existence of a Lyapunov function $V^\star$ such that Assumption \eqref{ass:cvSt} is satisfied. In practice, the candidate function $V$ that allows checking for the moment condition \eqref{lloglass} will be such that $V \log V=V^{\star}$ at infinity.
In effect, the proof only needs that condition \eqref{lloglass} be satisfied for $k$ large enough. Moreover, it is in general  sufficient to verify \eqref{lloglass} for $k=1$  by propagating the estimates along generations, see forthcoming Proposition~\ref{Propos-propag}.  This moment condition  is equivalent to the classical $L \log L $ criterion of Kesten-Stigum for Galton-Watson processes with a single type or a finite number of types. More precisely, condition \eqref{lloglass} is equivalent to {Asmussen-Hering} condition \eqref{LlogLh} as soon as it is bounded. 
Besides, given that $\gamma(V^{\star})<\infty$,  condition \eqref{lloglass} implies
\begin{equation}
\label{othermoment}
\E_{\gamma}(Z_1(V)\log^{\star} Z_1(V))<\infty,
\end{equation}
which in turn, since $h$ is dominated by $V$, implies the {Asmussen-Hering}
condition \eqref{LlogLh}.
With regard to applications,  we believe that conditions \eqref{lloglass} and \eqref{othermoment} are very close. More precisely, to the authors' knowledge and as applications may show, checking \eqref{othermoment} in practice often amounts to proving \eqref{lloglass} via drift conditions. The fact that condition \eqref{LlogLh} involves only the eigenelements $h$ and $\gamma$ is closely linked to the uniformity, with respect to $h$ and $\gamma$, in \eqref{UCsg}.  Let us also mention that spine techniques initiated in \cite{lyons1995conceptual, kurtz1997conceptual} allow us to see {non-degeneracy of the limit martingale} through boundedness of branching with immigration. With this approach, \cite{athreya2000change} has proved non degeneracy of  {the limit of the martingale} $W$ under uniform $L\log L$ condition for general state spaces. 
{On the non-degeneracy of the martingale under the $L \log L$ condition, we can cite the very sharp results of \cite{liu2011log}, for local branching processes whose law is density-based, and \cite{2025arXiv250305575A} who obtained very recently sharp results on countable spaces.}
Moreover, with this approach, \cite{zbMATH07199305} have obtained a strong law of large numbers, beyond our (positive) recurrent framework,  in the $L^2$ case with partial relaxation of the  uniform convergence of the first moment semigroup. \\
The approach here, involving decomposition of the empirical measure with family of martingales and contractions, is complementary to the {aforementioned ones}. In the finite dimensional case and the uniform case, condition \eqref{UCsg} of \cite{asmussen1976strong} amounts, in our setting, to $h=V=V^{\star}$, whereby then, all functions spaces involved in the proofs coincide.  The case $h\ll  V=V^{\star}=1$  will be useful for applications in Section \ref{sec:Bdiffusion} and \ref{sec:HouseOfCards}, while 
 the case  $V$ unbounded and $h\leq V \ll V^{\star} \sim V \log^{\star}(V)$ will   be relevant 
in Section \ref{GFMod} and application to growth fragmentation.  Indeed, roughly speaking,
$Z_1(V)\log^{\star} Z_1(V)$ is expected to be 
of order $V(x)\log^{\star} V(x)$ when the traits of the offspring $Z_1$ are comparable to the parent trait $x$ with positive probability. Note that we do not particularly assume that $1 \leq V$ and {hence we can capture} examples where we may not be able to control the mean number of individuals. \\

The proposed applications relate to our original motivations for this work and are in continuous-time, namely branching diffusions with absorption, house of cards model and growth fragmentation. Therefore, we  provide a  continuous-time framework which allows to construct general branching processes, check non explosion as well as condition \eqref{lloglass}, while the semigroup behavior \eqref{ergodnc} will be derived from recent works.\\
The main ingredients of our proofs could be extended to other branching structures and in particular adapted to the study of superprocesses, see e.g. \cite{englander2009law}. We note that such extensions may require replacing individual based decompositions with analytic martingale measure framework, used in \cite{liu2013strong}, together with resolvent estimates. This could lead to interesting developments.  

The paper is organized as follows. In Section~\ref{se:discret}, we focus on the discrete-time framework. We first give a general construction of the branching process in Section \ref{sec:construction} and then introduce the contraction operator as well as the key decomposition of the renormalized empirical measure  $(Z_n(f)/\lambda^n)_n$ in Section \ref{sec:contraction-and-LlogL}. In Section~\ref{sec:fam-of-mg}, we study the families of martingales involved in this decomposition.  We then prove the main result in discrete-time and provide some additional estimates in Section \ref{sec:main-discret}, which are useful in particular for continuous-time.  The last section, Section \ref{sec:continuous-time-and-examples}, is devoted to the continuous framework. From the discrete framework, we derive in Section \ref{sec:gen-cv-result} the non-degenerescence of the {the limit of the martingale} and the a.s.~and $L^1$ convergence of the renormalized empirical measure. Finally in Sections \ref{sec:Bdiffusion} and \ref{GFMod}, we present two applications.

\section{General study in discrete-time}
\label{se:discret}
\subsection{Construction and  definitions}\label{sec:construction}

We proceed now with the construction of discrete-time Markov branching processes with measurable state space $\mathcal X$. We also give a few details on the branching  property and the associated first moment semigroup. 

Let $\mathcal X$ be a measurable space endowed with the $\sigma$-field  $\mathcal B_{\mathcal X}$ and $(\Omega, \mathcal F, \mathbb P)$ a probability space satisfying the usual conditions. We use Ulam-Harris-Neveu notations, {where $\varnothing$ denotes the root individual} and let 
\begin{align*}
    \mathcal{U}=\bigcup_{n\geq 0} \N^n=\{\varnothing\} \cup \N \cup \N^2 \ldots
\end{align*}
be the set that allows to label individuals whilst retaining the genealogical information. Each individual of trait $x \in \mathcal{X}$  independently gives birth to a random number $L^x$ of progeny, whose law depends on $x$. The trait distribution of the offspring is also trait-dependent. More precisely, trait transmission to children is modelled by a family of random vectors $(\Theta_x)_{x\in \mathcal{X}},$  whereby for any $x\in \mtcX,$ $ \Theta_x=(X_i^x: i=1,\ldots, L^x)$ is a r.v.~valued in $\mathbb{X}=\cup_{k\geq 0} \mathcal X^k$, endowed with the $\sigma$-field $\cup_{k\geq 0} (\mathcal B_{\mathcal X})^k$. The r.v.~$X_i^x$ represents the trait of the $i$th child of an individual whose trait is given by $x$.
Let us now choose independently for each label 
$u\in \mathcal U$ a family of random vectors $(\Theta_x(u))_{x\in \mtcX}$ with the common  distribution $(\Theta_x)_{x\in \mtcX}$.

We denote $\Theta_x(u)=(X_i^x(u): i=1,\ldots, L^x(u))$ where $X_i^x(u)$
is the trait of 
the $i$th offspring of individual $u$ with 
trait  $x$ and $L^x(u)$ the number of offspring.  Finally, we  assume that  $\Theta, \Theta(u) : \mathcal X\times \Omega \rightarrow \mathbb X $ 
are measurable, for $u\in \mathcal U$, where  $ \mathcal X\times \Omega $ is endowed with the natural product $\sigma$-field. 
We can now define recursively the branching process, by constructing simultaneously the set of living individuals and their trait. We start with one single individual in generation $0$.
\begin{definition}
    The branching process started with one single individual  of trait $x_0\in \mathcal X$ and  with  reproduction r.v.~$(\Theta_x(u))_{x\in \mtcX, u \in \mtcU}$ is the family of r.v.~$(Z(u))_{ u\in \mathbb G_n, n\in \mathbb N}$ 
     defined recursively by 
    \begin{align*}
        \mathbb{G}_0=\{ \varnothing\}, \quad Z(\varnothing)=x_0,
    \end{align*}
    and for all $n\geq 0,$
    \begin{align*}
        \mathbb{G}_{n+1}=\{ uk: u \in \mathbb{G}_n, k \in \llbracket 1, L^{Z(u)}(u) \rrbracket\}, \quad \quad  Z(uk)=X_k^{Z(u)}(u),
    \end{align*}
    for all $u \in \mathbb{G}_n$ and $k \in \llbracket 1, L^{Z(u)}(u) \rrbracket.$
\end{definition}
In this definition, $\mathbb G_n$ is the set of individuals in generation $n$ and $Z(u)$ the trait of individual $u$.
This definition implies that the law of $(Z(u))_{ u\in \mathbb G_n, n\in \mathbb N}$ is determined by the law of $(\Theta_x)_{x\in \mtcX}$ and by the initial trait value $x_0.$ Observe that this definition easily extends, by the branching property, to any type of initial conditions, including several individuals and any random traits.

The process
$$Z_n=\sum_{u\in \mathbb G_n} \delta_{Z(u)}$$
so constructed verifies both a Markov property and a branching property, which are directly inherited from the independence of the r.v.~$(\Theta_x(u))_{x\in \mtcX}$ for $u\in \mtcU$. More precisely, for any $n\geq 0,$ we let $\mtcF_n$ be the natural filtration to $Z_n.$ Conditional on $\mtcF_n$, the processes
\begin{equation}
\label{rooted}
Z^{(u)}=(Z^{(u)}_p)_{p\in \N}, \quad \text{where } \quad Z^{(u)}_p= \sum_{ v : uv\in \mathbb G_{n+p}} \delta_{Z(uv)}
\end{equation}
are independent and $Z^{(u)}$ is distributed as the original process $(Z_n)_{n\in \mathbb N}$ started from  $\delta_{Z(u)}$.

For all $x\in \mathcal{X}$ and $f$ measurable and non-negative function on $\mathcal{X}$,  we define
$$Sf(x)=\E_{\delta_x}\left(\sum_{u\in \mathbb G_1}f(Z(u))\right)=\E_{\delta_x}(Z_1(f)),$$
which is non-negative but can be infinite. Similarly, for $n\geq 0,$ we define the non-negative and possibly infinite quantity
$$S_nf(x)=\E_{\delta_x}\left(\sum_{u\in \mathbb G_n}f(Z(u))\right)=\E_{\delta_x}(Z_n(f)).$$
We have that $S_n= S_1 \circ \cdots \circ S_1 =S^n$, $S_1=S$ and $S_0=\textsc{Id}$. We assume that there exists a positive function $V^\star$ such that there exists $C>0$ which satisfies
\begin{equation}
\label{eq:SVleqCV}
SV^\star(x) \leq C V^\star(x),
\end{equation}
for any $x\in \mathcal X$. We observe that  Assumption~\ref{ass:cvSt} is stronger than \eqref{eq:SVleqCV} by  taking $n=1$ and relying on the facts that $h$ is dominated by $V^\star$ and $\gamma(V^\star) < \infty$. Inequality \eqref{eq:SVleqCV} guarantees that the space $\mathcal B_+(V^\star)$ of measurable non-negative functions $f$ such that
$$\sup_{x\in \mtcX} \frac{f(x)}{V^\star(x)} <\infty,  $$
is {stable under} $S$. More precisely, if
$f \in \mathcal  B_+(V^\star)$, then $Sf\in \mathcal  B_+(V^\star)$ and by iteration $S_nf\in \mathcal  B_+(V^\star)$ for any $n\in \N.$

We can next consider the space
$\mathcal B(V^\star)$ of measurable  functions $f$ such that
$$\sup_{x\in \mtcX} \frac{\vert f(x) \vert}{V^\star(x)} <\infty,$$
and extend the definition of $S$ to this space by using the positive and negative parts of $f \in \mtcB(V^{\star})$:
$$Sf=Sf_+-Sf_-.$$
The space $\mathcal{B}(V^\star)$ is also {stable under} $S$ and  $S$ enjoys the semigroup property on $\mathcal{B}(V^\star)$:
\begin{align*}
S_{n+1}f=S_n(Sf)=S(S_nf),
\end{align*}
for any $f\in \mathcal{B}(V^\star)$.
\subsection{Contraction and $L\log L$ moment}\label{sec:contraction-and-LlogL}

For an integer $r$, we introduce the operator $T_r$ through
$$T_rf=\lambda^{-r} S_rf -\gamma(f)h,$$
and its $n$-th iteration 
$$T^{n}_r=T_r\circ \cdots  \circ T_r,$$
with the convention $T_r^0=\textsc{Id},$ for the identity operator. The space $\mathcal{B}(V^\star)$ of measurable functions $f:\mtcX\to \R$ which are dominated by $V^\star,$ endowed with the norm $\lVert \cdot \rVert_{\mtcB(V^\star)}$ defined by
\begin{align*}
   \forall f \in \mtcB(V^\star), \quad  \lVert f\rVert_{\mtcB(V^\star)}:=\sup_{x\in \mtcX}\frac{\vert f(x)\rvert}{V^\star(x)},
\end{align*}
is a Banach space. Assumption \ref{ass:cvSt} on the semigroup writes
$$ \lVert T_r f\rVert_{ \mtcB(V^\star)}\leq a_r \lVert f\rVert_{\mtcB(V^\star)},$$
and implies  for $n\geq 0$,
\begin{equation}
\label{cacontracte}
 \lVert T^{n}_r f\rVert_{\mtcB(V^\star)}\leq a_r^n
\lVert f\rVert_{\mtcB(V^\star)}.
\end{equation}
When $a_r$ is smaller than $1$, we obtain a contraction. In practice,  we will use the following exponential decrease 
$$\vert T^{n}_r f(x) \vert \leq a_r^n V^\star(x), $$
that holds for any $\vert f\vert \leq V^\star$, $x\in \mathcal X$ and $n\geq 0$. 
 Moreover $\gamma$ is a left eigenmeasure for $S$ and for any $r\geq 1,$ we will also use  $ \gamma T_r=0$.
We study the renormalized empirical measure 
 $$X_n^{(r)}:=\lambda^{-nr}Z_{nr}.$$ Note that $(X_n^{(1)}(h))_n=(Z_{n}(h)/\lambda^{n})_n$ is the classical non-negative martingale associated to the harmonic function $h$. Our target is the process $X_n^{(1)}(f)$ whose asymptotic behavior we wish to describe.
Recall that $(\mathcal F_n)_n$ is the natural filtration of the branching process $(Z_n)_n$ and  denote $\mathcal F_n^{(r)}=\mathcal F_{nr}$ for $n\geq 0,r\geq 1$.


The  proofs involve  the martingale increments, defined for any $f\in \mathcal{B}(V^\star)$  by
$$\Delta_n^{(r)}(f):=X_{n}^{(r)}(f)-\E(X_{n}^{(r)}(f) \, \vert \, \mathcal F_{n-1}^{(r)})=M_n^{(r)}(f)-M_{n-1}^{(r)}(f),$$ 
and the associated martingale starting from $0$ satisfying for $n\geq 1$ \begin{align}
\label{defmart}
M_n^{(r)}(f)=\sum_{i=1}^n \Delta_i^{(r)}(f).
\end{align}
Taking $f=h,$  $X_n^{(r)}(h)=X_0^{(r)}(h)+M_n^{(r)}(h)$ and we recall that it is a martingale. 
Our approach relies  on  the following  decomposition of $X_n^{(r)}(f)$ for more general functions $f$ in $\mathcal B(V^{\star})$. It allows to exploit the previous family of martingales together with contraction properties of $T$.
\begin{Lemma} \label{dec} Under Assumption \ref{ass:cvSt}, for any $f\in \mtcB(V^\star)$, $r\geq 1$ and {$n\geq 1$},
\begin{align*}
    X_n^{(r)}(f)=X_0^{(r)}(T^n_r f)+\gamma(f) \, X^{(r)}_{n-1}(h)\, + \, R_n^{(r)}(f),
\end{align*}
where
\begin{align*}
R_n^{(r)}(f)&=\sum_{i=1}^{n} \Delta_{i}^{(r)}(T^{n-i}_{r}f)=
M_n^{(r)}(f)+\sum_{i=1}^{n-1} M_{n-i}^{(r)}(T^{i-1}_{r}(T_r-\textsc{Id})f).
\end{align*}
 \end{Lemma}
We remark that in the forthcoming proof, all that is actually needed from Assumption \ref{ass:cvSt}, is that $\gamma T_r=0$.
 \begin{proof} We use for $f\in \mathcal{B}(V^\star)$,
 \begin{align*}
    X_{n}^{(r)}(f)&=\E(X_{n}^{(r)}(f) \, \vert \, \mathcal F_{n-1}^{(r)})+ \Delta_n^{(r)}(f).
\end{align*} 
Moreover, by decomposing the population of generation $(n+1)r$ in terms of the ancestors belonging to generation $nr$ and writing $v\succ u$ when $v$ is a descendant of $u$,
\begin{align*}
\E(X_{n+1}^{(r)}(f) \, \vert \, \mathcal F_{n}^{(r)})&=\lambda^{-nr}\sum_{u\in \mathbb G_{nr}} \E\left(\lambda^{-r}  \sum_{v \in \mathbb G_{(n+1)r} : v\succ u} f(Z(v))  \vert \, \mathcal F_n^{(r)}\right)\\
&=\lambda^{-nr}\sum_{u\in \mathbb G_{nr}} \lambda^{-r}S_rf(Z(u))\\
&=X^{(r)}_n(\lambda^{-r}S_rf)=X^{(r)}_n(T_rf)+\gamma(f) \, X^{(r)}_n(h).
\end{align*}
By iterating this identity, we obtain for $n\geq 1$,
\begin{align*}
    X_n^{(r)}(f)=X_0^{(r)}(T_r^nf)+\sum_{i=1}^n \gamma(T_r^{i-1}f) \, X^{(r)}_{n-i}(h)\, + \, R_n^{(r)}(f),
\end{align*}
where 
\begin{equation}
\label{martincrM}
 R_n^{(r)}(f)=
    \sum_{i=1}^{n} \Delta_{i}^{(r)}(T^{n-i}_rf)= \sum_{i=1}^{n} M_{i}^{(r)}(T^{n-i}_rf)-M_{i-1}^{(r)}(T^{n-i}_rf).
    \end{equation}
The result follows by recalling that $\gamma T_r=0$, which implies $\gamma(T_r^{i-1}f)=0$ for $i\geq 2$, and rearranging the last sum by linearity of $f\mapsto M_i^{(r)}(f)$.
 \end{proof}
Before studying the martingales involved in this decomposition, 
we give a useful condition to ensure that the moment condition \eqref{lloglass} on $Z_n(V)\log^{\star} Z_n(V)$   propagates for $n\geq 1$. It will also guarantee that it grows  
like the first eigenvalue $\lambda$, at the  logarithmic scale. These conditions will be satisfied in each of our applications. 
\begin{Prop}\label{Propos-propag}
Let Assumption \ref{ass:cvSt} hold and $V:\mathcal{X} \to (0,\infty)$ be a measurable function such that $V\leq V^\star$ and $V\log(V)\in \mathcal B(V^{\star})$. We
consider for  $n\geq 0$, 
$$I_n= \sup_{x \in \mathcal X} \frac{\E_{\delta_x}(Z_n(V)\log^{\star} Z_n(V))}{V^\star(x)}.$$
(i)  If   $I_1<\infty$ and $ SV \in \mathcal{B}(V)$, then $I_n<\infty$ for any $n  \geq 1$.\\
(ii) If $I_1<\infty$ and there exists $C\geq 0$ such that for any $n\in \mathbb N$, $S_nV\leq C \lambda^n V$,
then $I_n<\infty$ for any $n \in \N^*$ and  
$$\lim_{n\rightarrow \infty} \frac{\log I_n}{n}=\log \lambda.$$
\end{Prop}
\begin{proof}
The assertion (i) is a direct consequence of forthcoming Lemma~\ref{Lemma-propag-i} in Appendix, with $m=n,\, q=1, \, r=0$.

 We now turn to the asymptotic analysis of  $\log I_n/n$. A lower-bound of this latter quantity can be derived by applying the semigroup Condition \eqref{ergodnc} to $V.$ Indeed:
\begin{align*}
    \frac{ \E_{\delta_x} (Z_{n}(V))}{V^\star(x)} =\frac{S_nV(x)}{V^\star(x)}\geq \lambda^{n}\left(\gamma(V)\frac{h(x)}{V^\star(x)}-a_n\right). 
\end{align*}
Then, considering the event where $Z_{n}(V)>e$ and its complementary set, we get
\begin{align*}  \frac{\E_{\delta_x}\left(Z_{n}(V)\log^{\star} Z_{n}(V)\right)}{V^\star(x)}\geq \frac{ \E_{\delta_x} (Z_{n}(V))}{V^\star(x)} -\frac{e}{V^\star(x)},
\end{align*}
and
\begin{align*}
I_n\geq \sup_{x\in \mathcal X}  \left\{\lambda^{n}\left(\gamma(V)\frac{h(x)}{V^\star(x)}-a_n\right) -\frac{e}{V^\star(x)} \right\}.
\end{align*}
We fix now $x$ on the right hand side and   using that $\limsup_{n\rightarrow \infty} \log(a_n)/n=0$  since $\sum a_n/n<\infty$. It yields $$\liminf_{n\rightarrow \infty} \log (I_n)/n\geq \log(\lambda).$$
Let us prove the upperbound and conclude.
Using  Lemma~ \ref{Lemma-propag-ii}, there exists a constant $0<C<\infty$ such that for any $n\in \N^*$ where $n=mq+r,$ with $m,q\in \N^*$ and $ 0\leq r\leq m-1$, 
\begin{align}\label{eq:I_n-leq-I_r-I_q}
      I_n \leq n \lambda^n C^m (I_r+I_q+1).
\end{align}
Observe first that this inequality applied to $m=n$, $q=1$ and $r=0$ ensures that
 \begin{align*}
    \limsup_{n\to \infty} \frac{\log I_n }{n}\leq \log \lambda +\log C:=C'.
\end{align*}
In the second step, we use again inequality \eqref{eq:I_n-leq-I_r-I_q} and choose  $m:=\lfloor \sqrt{n} \rfloor$ and $r, q \leq \lfloor \sqrt{n} \rfloor$ such that $n=\lfloor \sqrt{n} \rfloor q+r.$  Inequality \eqref{eq:I_n-leq-I_r-I_q} becomes
\begin{align*}
      I_n \leq n\lambda^n C^{\lfloor \sqrt{n} \rfloor} (I_r+I_q+1).
\end{align*}

We combine the last two bounds and the fact that $C<\infty$, $\gamma(V^\star)<\infty$
and get
\begin{align*}
    \limsup_{n \to \infty}   \frac{\log I_n}{n} \leq \log \lambda +\limsup_{n \to \infty} \frac{\log(C)\sqrt{n}+\log(n)}{n} +2C' \limsup_{n \to \infty} \frac{\sqrt{n}}{n}= \log \lambda,
\end{align*}
which ends the proof.
\end{proof}

\subsection{Families of martingales}\label{sec:fam-of-mg}
We now need to determine the long time behaviour of the martingales $M^{(r)}(f)$ and to finely control them. This is achieved by exploiting a martingale decomposition, respectively, into an $L^1$ and an $L^2$ contributions. Classically, the $L^1$ part gathers large jumps. The decomposition  here is inspired by the subtle truncation argument put forward by \cite{asmussen1976strong}. More precisely, the martingale increment writes for $f\in \mathcal{B}(V^\star)$,
\begin{align} 
\Delta_{n+1}^{(r)}(f)&=X_{n+1}^{(r)}(f)-\E(X_{n+1}^{(r)}(f) \, \vert \, \mathcal F_n^{(r)}) \nonumber\\
&=\lambda^{-nr}\sum_{u\in \mathbb G_{nr}}
\left\{\lambda^{-r}Z^{(u)}_r(f)-\E_{\delta_{Z(u)}}(X^{(r)}(f))\right\} \label{decDelta}
=A_{n+1}^{(r)}(f)+B_{n+1}^{(r)}(f),
 \end{align} 
where $Z^{(u)}$ is the branching process rooted in $u$, which has been defined in \eqref{rooted},
and the two contributions $A$ and $B$ are given for any $n\in \N$ by
 \begin{align} 
A_{n+1}^{(r)}(f)&=\lambda^{-nr}\sum_{u\in \mathbb G_{nr}} \left\{ \lambda^{-r}Z^{(u)}_r(f)\mathbf{1}_{Z^{(u)}_r(V)\leq \lambda^{nr}} -\E_{\delta_{Z(u)}}(X^{(r)}(f)\mathbf{1}_{Z_r(V)\leq \lambda^{nr}}) \right\},\label{def:A}\\
B_{n+1}^{(r)}(f)&=\lambda^{-nr}\sum_{u\in \mathbb G_{nr}} \left\{\lambda^{-r}Z^{(u)}_r(f)\mathbf{1}_{Z^{(u)}_r(V)>\lambda^{nr}} -\E_{\delta_{Z(u)}}(X^{(r)}(f)\mathbf{1}_{Z_r(V)>\lambda^{nr}})\right\}.\label{def:B}
 \end{align}
The convergence of the renormalized empirical measure will rely on the following convergences.
\begin{Prop}\label{desmartingales} 
Under the Assumptions of Theorem \ref{th:main-discret},    for any $f \in \mathcal B(V)$ and $r\geq 1$ and  $n \in \N$, the following decomposition holds
$$M^{(r)}_n(f)=\sum_{i=1}^n A_{i}^{(r)}(f)+\sum_{i=1}^n B_{i}^{(r)}(f),$$
where  $(\sum_{i=1}^n A_{i}^{(r)}(f))_n$ is a martingale, bounded in $L^2$  and 
$(\sum_{i=1}^n B_{i}^{(r)}(f))_n$ is a uniformly integrable martingale.\\
As a consequence, when $n\rightarrow \infty$, $(M^{(r)}_n(f))_n$  converges a.s. and in $L^1$ to $M^{(r)}_{\infty}(f)$  and $(X_n(h))_n$ converges a.s. and in $L^1$ to  $W\in [0,\infty)$ which satisfies $\E_{\delta_x}(W)=h(x)$.
  
\end{Prop} 

\begin{proof}
We first focus on the $L^1$ part involving $\sum\limits_{n\geq 0}B_n^{(r)}(f).$ We observe that for any $f\in \mathcal B(V)$,
 \begin{equation}
 \label{borne}
    \vert B_{n}^{(r)}(f) \vert  \leq  \, \lVert f\rVert_{\mtcB(V)} \,   B^{(r)+}_{n},  
 \end{equation}
  where 
  $$  B^{(r)+}_{n+1}:=\lambda^{-nr}\sum_{u\in \mathbb G_{nr}} \left\{\lambda^{-r}Z^{(u)}_r(V)\mathbf{1}_{Z^{(u)}_r(V)>\lambda^{nr}} +\E_{\delta_{Z(u)}}(X_r(V)\mathbf{1}_{Z_r(V)>\lambda^{nr}})\right\}.$$
Moreover
\begin{align*}
  \E_{\delta_x}\left(  
  B^{(r)+}_{n+1}   \right)&= 2   \lambda^{-(n+1)r}\E_{\delta_x}\left(  \sum_{u\in \mathbb G_{nr}}
  \phi^{(r)}_n(Z(u))    \right)
  =2\lambda^{-r}\lambda^{-nr}
S_{nr}\phi^{(r)}_n(x)\\
& \qquad \qquad \qquad \qquad \leq 2\lambda^{-r}\left(h(x)\gamma (\phi^{(r)}_n) + V^\star(x) a_{nr} \lVert \phi^{(r)}_n\rVert_{\mtcB(V^\star)}\right), 
\end{align*}
where for any $y\in \mathcal{X},$ 
\begin{equation}
\label{definature}
    \phi^{(r)}_n(y):=  \E_{\delta_y} 
(Z_r(V) \mathbf{1}_{Z_r(V)>\lambda^{nr}}).
\end{equation}
Next, as $h$ is dominated by $V^\star$ we obtain
\begin{align*}
& \sum_{n\geq 0} \E_{\delta_x}\left(  
  B^{(r)+}_{n+1}   \right) \\
&\quad \leq 2 \left(\lVert h \rVert_{\mtcB(V^\star)} +1\right) \lambda^{-r} V^\star(x) \left\{ \E_{\gamma}\left(\sum_{n\geq 0} Z_r(V) \mathbf{1}_{ Z_r(V)>\lambda^{nr}}\right)+
\sum_{n\geq 0} a_{nr} \lVert \phi^{(r)}_n\rVert_{\mtcB(V^\star)}\right\}.
\end{align*}
In the rest of this proof and subsequent proofs, we will denote $C$ a constant independent of $n, i,r \in \N$ but whose value can change from line to line. 
Recalling \eqref{definature}, Markov  inequality allows us to upper-bound this second expectation term:
$$
\lVert \phi^{(r)}_n\rVert_{\mtcB(V^\star)}\leq
\sup_{y \in \mathcal X}  \frac{1}{V^\star(y)} \E_{\delta_y}\left(Z_r(V)\frac{\log^\star Z_r(V)}{ \log^\star \lambda^{nr}}\right)\leq \frac{C}{nr \log \lambda}I_r,
$$
since $\log^{\star}\lambda^u\geq C u\log \lambda.$ 
To upper-bound the first expectation term, we write
$$N_r=\sup\{ n \in \mathbb N : Z_r(V) > \lambda^{nr}\}=\sup\{ n \in \mathbb N : \frac{\log^{\star}(Z_r(V))}{\log^\star \lambda^{n r}} > 1\},$$ 
which leads to
\begin{align*}
\sum_{n\geq 0} Z_r(V) \mathbf{1}_{ Z_r(V)>\lambda^{nr}}&\leq 
\sum_{n\geq 0} Z_r(V)\frac{\log^\star(Z_r(V))}{ \log^\star(\lambda^{N_rr})} \mathbf{1}_{ Z_r(V)>\lambda^{nr}}\\
&\leq 
Z_r(V) \frac{\log^\star(Z_r(V))}{ \log^\star(\lambda^{N_rr})} N_r
\leq  \frac{C}{r\log \lambda} Z_r(V)\log^\star(Z_r(V))
\end{align*}
and
$$\E_{\gamma}\left(\sum_{n\geq 0} Z_r(V) \mathbf{1}_{ Z_r(V)>\lambda^{nr}}\right)\leq   \frac{C}{r}  \gamma(V^\star) I_r.$$
Finally, gathering these estimates, we obtain
\begin{align*}
 \sum_{n\geq 0} \E_{\delta_x}\left(
  B^{(r)+}_{n+1}\right)  
  & \leq C V^\star(x) \frac{\lambda^{-r}}{r} I_r \Big(\gamma(V^\star)+\sum_{n\geq 0}\frac{a_{nr}}{n}\Big).
\end{align*}
These computations show that 
\begin{align}
\label{domBL1ps}
 \sum_{i=1}^{n}  
\vert  B_i^{(r)}(f)\vert \, \leq  \, \lVert f\rVert_{\mtcB(V)} \,\sum_{i=1}^{\infty}
    B^{(r)+}_i \qquad \text{ where} \qquad \sum_{i=1}^{\infty} 
    \E_{\delta_x}(B^{(r)+}_i)<\infty,
    \end{align}
which ensures that $(\sum_{i=1}^n B_{i}^{(r)}(f))_n$ is uniformly integrable.


Let us now turn towards the $L^2$ part of the decomposition of $M_n^{(r)}(f).$  We are  using the 
semi-martingale decomposition of the square of $\left(\sum_{i= 1}^n A_i^{(r)}(f)\right)_n$. More precisely,   we rely on  the quadratic variation and  use that
$$\left( \sum_{i= 1}^n A_i^{(r)}(f)\right)^2- \sum_{i=1}^n  \, \E\left( A_{i}^{(r)}(f)^2 \,  \vert \, \mathcal F^{(r)}_{i-1} \right) $$
is a martingale starting from $0$.
Besides, expanding the sum and using that the r.v. ~in  $A^{(r)}_{i+1}(f)$ are centred and independent conditionally on $\mathcal F^{(r)}_{i}$, we get:
 \begin{align}
 \label{idvq}
 \E\left( A^{(r)}_{i+1}(f)^2  \,  \vert \, \mathcal F^{(r)}_{i} \right)
=\lambda^{-ir} X_i^{(r)}(\psi^{(r,f)}_i),
 \end{align}
where for all $y\in \mathcal{X},$
$$\psi^{(r,f)}_i(y):=\E_{\delta_y} \left(\left(X_r(f)\mathbf{1}_{Z_r(V)\leq \lambda^{ir}} -\E_{\delta_y}(X_r(f)\mathbf{1}_{Z_r(V)\leq \lambda^{ir}}) \right)^2 \right).$$
Moreover 
$$\psi^{(r,f)}_i(y)\leq 2\E_{\delta_y} 
\left(X_r(f)^2 \mathbf{1}_{Z_r(V)\leq \lambda^{ir}} \right)\leq  \, \lVert f\rVert_{\mtcB(V)}^2 \,V^{(r)}_i(y),$$
with 
$$V^{(r)}_i(y)=2 \E_{\delta_y} 
\left(X_r(V)^2 \mathbf{1}_{Z_r(V)\leq \lambda^{ir}} \right).$$
We obtain
 \begin{align*}
\sum_{i=1}^n  \, \E\left( A_{i}^{(r)}(f)^2 \,  \vert \, \mathcal F^{(r)}_{i-1} \right)  & \leq \,  \lVert f\rVert_{\mtcB(V)}^2 \,\sum_{i= 0}^{\infty}  \lambda^{-ir} X_i^{(r)}(V^{(r)}_i).
 \end{align*}
To conclude, let us prove that the martingale is bounded in $L^2$ by computing the expectation of the right-hand side of the above inequality. 
  \begin{align*}
\E_{\delta_x}\left(\sum_{i=0}^{\infty}  \lambda^{-ir} X_i^{(r)}(V^{(r)}_i)\right)&=
\sum_{i=0}^{\infty}  \lambda^{-ir} \lambda^{-ir} S_{ir}V^{(r)}_i(x)\\
&\leq \sum_{i=0}^{\infty}  \lambda^{-ir}\left( h(x) \gamma(V^{(r)}_i)+V^\star(x)a_{ir} \lVert V^{(r)}_i\rVert_{\mtcB(V^\star)}\right).
  \end{align*} 
We follow similar arguments as with the $L^1$ part with the right-hand side above and observe that there exists a constant $C$ that does not depend on $r$ such that 
  $$\sum_{i=0}^\infty \lambda^{-ir} \mathbf{1}_{Z_r(V)\leq \lambda^{ir}} \leq \frac{C}{2Z_r(V)}.$$
  Then for any $y\in \mathcal X$, we write 
  \begin{align*}
\sum_{i= 0}^{\infty} \lambda^{-ir}V^{(r)}_i (y)&=2\E_{\delta_y} \left(X_r(V)^2 \sum_{i\geq 0} \lambda^{-ir} \mathbf{1}_{Z_r(V)\leq \lambda^{ir}} \right)\\
&\quad \qquad \leq C \lambda^{-r} \E_{\delta_y}( X_r(V))=C \lambda^{-2r}S_rV(y),
\end{align*}
which yields
\begin{align}
\label{abime}
    \sum_{i=0}^{\infty}  \lambda^{-ir} \gamma(V^{(r)}_i)\leq C\lambda^{-2r} \gamma (S_rV) {\leq} C\lambda^{-r} \gamma(V^\star),
\end{align}
where we used that $S_rV \leq \lambda^r V^\star(x).$
Observe as well for any $a,b>e$
$$a^2\mathbf{1}_{a\leq b}=a\log(a) \cdot \frac{a}{\log(a)}\mathbf{1}_{a/\log(a)\leq b/\log(b)}\leq a\log(a)\cdot\frac{b}{\log(b)},$$ since $u/\log(u)$ is increasing for $u>e$, so 
$$\E_{\delta_y}(Z_r(V)^2 \mathbf{1}_{e\leq Z_r(V)\leq \lambda^{ir}})\leq  \frac{\lambda^{ir}}{ir \log^\star \lambda} \, \E_{\delta_y}(Z_r(V)\log^\star(Z_r(V))).$$
Noting that $u^2=eu\log^{\star}(u)$ for $u<e$ and that there exists a constant $C$ depending only on $\lambda$ such that $\lambda^{ir} \geq Cir\log^{\star}\lambda$ for any $i,r\geq 1$, we get
$$\E_{\delta_y}(Z_r(V)^2 \mathbf{1}_{ Z_r(V)\leq \lambda^{ir}})\leq  \frac{C \lambda^{ir}}{ir \log^\star \lambda} \, \E_{\delta_y}(Z_r(V)\log^{\star}(Z_r(V))).$$
Adding that
$$\lVert V^{(r)}_i\rVert_{\mtcB(V^\star)}=\frac{2}{\lambda^{2r}}\sup_{y\in \mathcal X}  \frac{\E_{\delta_y} \left(Z_r(V)^2 \mathbf{1}_{Z_r(V)\leq \lambda^{ir}} \right)}{V^\star(y)}$$
yields 
  \begin{align}
  \sum_{i=0}^{\infty}  \lambda^{-ir} a_{ri} \lVert V^{(r)}_i\rVert_{\mtcB(V^\star)} \leq C \frac{\lambda^{-2r}}{\log\lambda} \, I_r \, \sum_{i=0}^{\infty}   \frac{a_{ri}}{ri}
  <\infty. \label{abim}
  \end{align}
 Gathering these estimates and recalling that $h$ is dominated by $V$
 by assumption, 
 we obtain
 \begin{align}
\E_{\delta_x}\left(\left(\sum_{i= 1}^n 
  A_i^{(r)}(f)\right)^2\right) & 
  =\E_{\delta_x}\left( \sum_{i=1}^n  \, \E\left( A_{i}^{(r)}(f)^2 \,  \vert \, \mathcal F^{(r)}_{i-1} \right)\right) \label{carreA}\\
 & \leq \lVert f\rVert_{\mtcB(V)}^2 \, \E_{\delta_x}\left(\sum_{i= 0 }^{\infty}  \lambda^{-ir} X_i^{(r)}(V^{(r)}_i)\right)\nonumber\\
  & \leq C \lVert f\rVert_{\mtcB(V)}^2 \,\lambda^{-r} V^\star(x)\left(\gamma(V^\star)+\lambda^{-r}I_r \, \sum_{i=0}^{\infty}   \frac{a_{ri}}{ri}\right)<\infty.\nonumber
 \end{align}
This shows the $L^2$ boundedness, and also provides a more quantitative estimate. The first part of the proposition has thus been proved.

We obtain directly the convergence of the martingale $(M^{(r)}(f))_n$, which is the sum of a uniformly integrable martingale and a martingale bounded in $L^2$. Indeed, 
a uniformly integrable martingale converges a.s.~and in $L^1$; see for instance \cite[Chapter 14, page134]{williams1991probability}.  Besides, a  martingale bounded in $L^2$ converges a.s.~and in $L^2$ (so in $L^1$ too). Recalling that
 $X_n(h)=h(x)+M_n^{(1)}(h)$ proves the last part.
 \end{proof}

Finally, besides these decompositions, we also need to control the series of martingales of Lemma \ref{dec}. The latter is achieved by exploiting the contraction $T_r$.
\begin{Lemma} \label{Restepetit} Under the Assumptions of Theorem \ref{th:main-discret}, for any $f \in \mathcal B(V)$ and any $r$ such that $a_r<1$, 
$$\sup_{n>N} \bigg\vert \sum_{i=N}^{n-1} M_{n-i}^{(r)}(T^{i-1}_r(T_r-\textsc{Id})f) \bigg\vert\stackrel{N\rightarrow \infty}{\longrightarrow} 0 \quad 
\text{and} \quad 
M_{\infty}^{(r)}(T_r^if)\stackrel{i\rightarrow \infty}{\longrightarrow} 0 $$
a.s. and in  $L^1.$
\end{Lemma}
\begin{proof} First,  for $n>N$, recalling the link between $M$ and $\Delta$ in \eqref{defmart}
\begin{align}
\sum_{i=N}^{n-1} M_{n-i}^{(r)}(T^{i-1}_r(T_r-\textsc{Id})f)&=\sum_{i=1}^{n-N}M_i^{(r)}(T_r^{n-i}f-T_r^{n-(i+1)}f)\nonumber\\
& = M_1^{(r)}(T_r^{(n-1)}f)+\sum_{i=2}^{n-N}(M_i^{(r)}-M_{i-1}^{(r)})(T_r^{n-i}f)-M_{n-N}^{(r)}(T_r^{N-1}f)\nonumber\\
& =\sum_{i=1}^{n-N}\Delta_i^{(r)}(Q_r^{n,i,N}f),\label{preuve-bord?}
\end{align}
whereby, to ease notations, we introduce the operator 
$$Q_r^{n,i,N}:= T_r^{n-i}-T_r^{N-1}.$$ 
Using the contraction $T_r$, 
it satisfies  for any $1\leq i\leq n-N$,
 $$\lVert Q_r^{n,i,N} f\rVert_{\mtcB(V^\star)} \leq 
 2a_r^{n-i}\lVert f\rVert_{\mtcB(V^\star)},$$
and we are  exploiting  decomposition \eqref{decDelta}
\begin{equation}
\label{decDelta2}
\sum_{i=1}^{n-N} \Delta_{i}^{(r)}(Q_r^{n,i,N}f)=\sum_{i=1}^{n-N}  A_{i}^{(r)}(Q_r^{n,i,N}f)+\sum_{i=1}^{n-N}  B_{i}^{(r)}(Q_r^{n,i,N}f).
\end{equation}
First, focusing on the $L^1$ part, we get the following bound
\begin{align*}
\sup_{n>N} \big\vert \sum_{i=1}^{n-N} B_{i}^{(r)}(Q_r^{n,i,N}f) \big\vert
&\leq 
\sum_{n>N} \sum_{i=1}^{n-N} \vert B_{i}^{(r)}(Q_r^{n,i,N}f)\vert \\
&\leq  \sum_{i=1}^{\infty} B_{i}^{(r)+}  \sum_{n\geq i+N} 2a_r^{n-i}\lVert f\rVert_{\mtcB(V^\star)}
\leq  \lVert f\rVert_{\mtcB(V^\star)} \, \frac{2a_r^N}{1-a_r}\sum_{i=1}^{\infty} B_{i}^{(r)+}, 
\end{align*}
where we recall \eqref{borne} and the $a_r$ contraction of $T_r$. This proves a.s. and $L^1$ convergence of the $L^1$ component involving $B$. 

Let us turn to the $L^2$ part. 
Recalling the computations of the previous proof (Proposition \ref{desmartingales}) involving  $\sum_{i=1}^{n}  A_{i}^{(r)}(f)$, we get
similarly
\begin{align*}
\E_{\delta_x}\left( \left(\sum_{i=1}^{n-N}  A_{i}^{(r)}(Q_r^{n,i,N}f)\right)^2 \right) &=  \sum_{i=1}^{n-N} 
\E_{\delta_x}\left( \left( A_{i}^{(r)}(Q_r^{n,i,N}f)\right)^2\right)\\
&\leq \sum_{i=1}^{n-N}\lVert Q_r^{n,i,N}f\rVert_{\mtcB(V^\star)}^2 \lambda^{-ir} \lambda^{-ir} S_{ir} V_{i}^{(r)}(x)\\
&\leq 4 \, \lVert f\rVert_{\mtcB(V^\star)}^2\, \sum_{i=1}^{n-N} a_r^{2(n-i)} \,   u_i(x),
\end{align*}
where 
$$u_i(x):= \lambda^{-ir}\left(h(x)\gamma(V_{i}^{(r)})+a_{ir}V^\star(x)\lVert V_{i}^{(r)} \rVert_{\mtcB(V^\star)} \right).$$
Now recalling that  $a_r<1$ and \eqref{abime} and \eqref{abim}, $\sum\limits_i u_i(x)<\infty$ and
$$\sum_{N} \sum_{n>N}\sum_{i=1}^{n-N} \, a_r^{2(n-i)} u_i(x)  = \sum_{i=1}^{\infty} u_i(x)\,\sum_N \sum_{k>N} a_r^{2k}<{\infty}.$$
Thus
 \begin{align*}
 \sum_{N} \sum_{n>N}  \E_{\delta_x}\left( \left(\sum_{i=1}^{n-N}  A_{i}^{(r)}(Q_r^{n,i,N}f)\right)^2\right) & <\infty.
\end{align*}
Using that
$$\sup_{n> N} \left(\sum_{i=1}^{n-N}  A_{i}^{(r)}(Q_r^{n,i,N}f)\right)^2\leq \sum_{n>N}  \left(\sum_{i=1}^{n-N}  A_{i}^{(r)}(Q_r^{n,i,N}f)\right)^2,$$
we get 
$$\sum_{N\geq 1} \E_{\delta_x}\left(\sup_{n> N} \left(\sum_{i=1}^{n-N}  A_{i}^{(r)}(Q_r^{n,i,N}f)\right)^2 \right) <\infty \quad \text{a.s.}$$  
This ensures both the following  $L^2$ convergence 
$$\E_{\delta_x}\left(\sup_{n> N} \left(\sum_{i=1}^{n-N}  A_{i}^{(r)}(Q_r^{n,i,N}f)\right)^2 \right)\stackrel{N\rightarrow\infty}{\longrightarrow} 0,  $$
 and  
$$\sum_{N\geq 1} \sup_{n> N} \left(\sum\limits_{i=1}\limits^{n-N}  A_{i}^{(r)}(Q_r^{n,i,N}f)\right)^2 <\infty  \quad \text{a.s.} $$
This latter convergence yields 
$$ \sup_{n> N} \left\vert \sum_{i=1}^{n-N}  A_{i}^{(r)}(Q_r^{n,i,N}f)\right\vert \stackrel{N\rightarrow\infty}{\longrightarrow} 0 \qquad \text{a.s.} $$
It ensures  the  a.s. and $L^1$ convergence to $0$ of the $L^2$ part involving $A$ in \eqref{decDelta2}. Recalling that these convergences hold also for the $L^1$ part represented by $B$ proves the first part of the lemma. \\
We now prove the second part of the lemma. To do so, let us recall the decomposition  $M_n^{(r)}(f)=\mathcal A_n^{(r)}(f)+\mathcal B_n^{(r)}(f)$ from proposition~\ref{desmartingales} where
$$\mathcal A_n^{(r)}(f):=\sum_{i=1}^n A_i^{(r)}{(f)} \stackrel{n\rightarrow \infty}{\longrightarrow} \mathcal A_{\infty}^{(r)}(f)\quad  \text{a.s. and in } L^2,$$
and 
$$\mathcal B_n^{(r)}(f):=\sum_{i=1}^n B_i^{(r)}{(f)} \stackrel{n\rightarrow \infty}{\longrightarrow} \mathcal B_{\infty}^{(r)}(f)
\quad \text{a.s. and in } L^1$$
and 
$$M_\infty^{(r)}(f)=A_{\infty}^{(r)}(f)+B_{\infty}^{(r)}(f).$$
We first prove that $\mathcal A_{\infty}^{(r)}(T_n^{(r)}f)$ goes to $0$ as $n$ tends to infinity, a.s. and in $L^2$, and then that  $\mathcal B_{\infty}^{(r)}(T_n^{(r)}f)$ goes to $0$ as $n$ tends to infinity, a.s. and in $L^1$. 
Indeed \eqref{carreA} ensures 
$$\E(\mathcal A_{\infty}^{(r)}(T_i^{(r)}f)^2)=\lim_{n\rightarrow\infty}\E(\mathcal A_{n}^{(r)}(T_i^{(r)}f)^2)\leq C_r \lVert T_i^{(r)}f\rVert_{\mtcB(V)}^2 \leq C_r a_r^i \lVert f\rVert_{\mtcB(V)}^2, $$
where $C_r$ is a finite constant. As $a_r<1$, $\mathcal A_{\infty}^{(r)}(T_i^{(r)}f)$ goes to $0$ in $L^2$. Moreover
$$\sum_{i\geq 1} \E(\mathcal A_{\infty}^{(r)}(T_i^{(r)}f)^2)<\infty, $$
which ensures that 
$$\sum_{i\geq 1}  A_{\infty}^{(r)}(T_i^{(r)}f)^2 <\infty \qquad \text{a.s.}$$
The previous convergence implies that $A_{\infty}^{(r)}(T_i^{(r)}f)$ tends a.s. to $0$ as $i$ tends to infinity.\\
Let us now turn to $\mathcal B_{\infty}^{(r)}$. We know  from \eqref{domBL1ps}
that
\begin{align*}
\vert\mathcal B_{\infty}^{(r)}(T_i^{(r)}f) \vert  \leq  \, \lVert T_i^{(r)} f\rVert_{\mtcB(V)} \,\sum_{j=1}^{\infty}
    B^{(r)+}_j \leq a_r^i \lVert f\rVert_{\mtcB(V)} \,\sum_{j=1}^{\infty}
    B^{(r)+}_j,
    \end{align*}
    where $\sum_{j=1}^{\infty} 
    \E_{\delta_x}(B^{(r)+}_j)<\infty$. This ends the proof by letting $i\rightarrow \infty$.
 \end{proof}

 \subsection{Proof of Theorem \ref{th:main-discret} and complements}\label{sec:main-discret}

We are now in a position to prove the desired results.
\begin{Prop}\label{prop:final_res-discrete} 
Under the Assumptions of Theorem \ref{th:main-discret}, for any $f \in \mathcal B(V)$ and for any $r \in \N$ such that $a_r<1$, the following convergence holds a.s.~and in $L^1$
\begin{align*}
\lim_{n\rightarrow \infty}    X_n^{(r)}(f)=\gamma(f) \, W.
\end{align*}
\end{Prop}
\begin{proof} We use Lemma 
 \ref{dec} to write
\begin{align*}
    X_n^{(r)}(f)=X_0^{(r)}(T_r^nf)+\gamma(f) \, X^{(r)}_{n-1}(h)\, + \, R_n^{(r)}(f),
\end{align*}
and  split the {remainder}  $R_n^{(r)}$ as
follows for any $N\geq 1.$
\begin{align*}
R_n^{(r)}(f)&=M_n^{(r)}(f)+\sum_{i=1}^{N-1} M_{n-i}^{(r)}(T^{i-1}_r(T_r-\textsc{Id})f)
+\sum_{i=N}^{n-1} M_{n-i}^{(r)}(T^{i-1}_r(T_r-\textsc{Id})f).
\end{align*}
We use Lemma \ref{Restepetit} 
to choose a random integer $N$ large enough so that the last term is small uniformly for any $n>N$ a.s. 
Similarly, we can choose $N$ (non random) large enough so that the $L^1$ norm of the last term is small enough. Then we use  Proposition \ref{desmartingales} to obtain the convergence of $M_n^{(r)}(f)$  to $M_{\infty}^{(r)}(f)$ and the convergence of $M_{n-i}^{(r)}(T^{i-1}_r(T_r-\textsc{Id})f)$ to $M_{\infty}^{(r)}(T^{i-1}_r(T_r-\textsc{Id})f)$ as $n$ tends to infinity, for $i\leq N$. Gathering these estimates yields
\begin{align*}
\lim_{n\rightarrow \infty}R_n^{(r)}(f)&=M_{\infty}^{(r)}(f)+\sum_{i=1}^{\infty} M_{\infty}^{(r)}(T^{i-1}_r(T_r-\textsc{Id})f)\\
&=M_{\infty}^{(r)}(f)+\sum_{i=1}^{\infty} M_{\infty}^{(r)}(T_{r}^if)-M_{\infty}^{(r)}(T^{i-1}_rf)=0 \qquad \text{a.s. and in } L^1,
\end{align*}
where we relied on telescoping sums exploiting that $M_{\infty}^{(r)}(T_r^if)$ goes to $0$  as $i$ tends to infinity from Lemma \ref{Restepetit}. Adding that when $n\rightarrow \infty$, $T_r^nf$ tends to $0$ and that $X^{(r)}_{n-1}(h)$ tends a.s. and in $L^1$ to $W$ (see Proposition \ref{desmartingales}) ends the proof.
\end{proof}
\begin{proof}[Proof of Theorem \ref{th:main-discret}]
We derive from this result the counterpart for $r=1$, by using the result with $X^{(r,k)}_n=X_{nr+k}$ for $0\leq k \leq r-1$ to cover the full set of integers.
Fix $r$ such that $a_r<1$. Starting from $Z_0=\delta_x$, Proposition \ref{prop:final_res-discrete} implies that $X^{r,k}_n(f)$ converges a.s~and in $L^1$ to $\gamma(f)W^{(r,k,x)}$. We now need to verify that the limits coincide for $k=0, \ldots, r-1$.
Indeed, considering  $f=h $ in this limit, $X^{(r,k)}_n(h)$ tends to $W^{(r,k,x)}$. Using the convergence of the  martingale $X^{(1)}_ n=Z_n(h)\lambda^{-n}$ towards $W^{(x)}$, we  can identify the limits  and conclude that $ W^{(r,k,x)}=W^{(x)}$ a.s. It yields
\begin{align*}
\lim_{n\rightarrow \infty}    X_n(f)=\gamma(f) \, W \quad  \text{a.s. and in } L^1
\end{align*}
for $f\in \mathcal B(V)$ and ends the proof of Theorem \ref{th:main-discret}.
\end{proof}

Let us mention that the previous results  of Proposition \ref{prop:final_res-discrete} and Theorem \ref{th:main-discret} provide bounds on $\mathbb P_{\delta_x}(W>0)$ and comparison to function $h,V, V^{\star}$, using in particular Paley–Zygmund inequality. But, up to our knowledge and as counterexamples have shown \cite{braunsteins2019pathwise,2025arXiv250305575A}, the event $\{W>0\}$ may not coincide with the survival event $\{\forall n\geq 0, \mathbb G_n\ne \emptyset\}$. However, in general, for some measurable set $A$ such that $\gamma(A)>0$, the  survival event  $\{\limsup_n Z_n(A)>0\}$  is included (and thus equal to) the local divergence event $\{\limsup_n Z_n(A)=\infty\}$ in the supercritical case. Therefore, proving that  $\mathbb P_{\delta_x}(W>0)$  for $x\in A$ is sufficient to conclude that $W>0$ on this event.  The discrete case is particularly simple since $A$ can be taken as a singleton, see e.g.  \cite{bansaye2023growth} for an example.


Before moving onto the continuous-time setting, we complement the discrete-time results with the following extension of convergences. We now associate to individuals more than a trait in $\mathcal X$. More precisely, for each individual $u$, we associate a random variable $Box(u)$. It is of the form $Box(u)=(Z(u),W(u))$, where $Z(u)\in \mathcal X$ is the trait of $u$ as before and $W(u)$ is a r.v. taking values in $\mathcal W$, so that  $Box(u)$  takes values in  a larger measurable state space $\overline{\mathcal X}=\mathcal X \times \mathcal W$. This extension of the trait space will not impact the original branching process $X$, but add useful information on the population. Such construction is linked to branching with characteristics and studies of Crump Mode Jagers processes \cite{zbMATH03555176}. It will be useful to prove the results in continuous-time.    We denote by $(\overline{\mathcal F}_n)_n$ the corresponding filtration, which extends  filtration $({\mathcal F}_n)_n$ :
$$\overline{\mathcal F}_n=\sigma(Box(v), Z(u) : v \prec u , u   \in \mathbb G_n) \supset \mathcal F_n. $$
We require that it verifies the following Branching-Markov type assumption : for any measurable non-negative function $F$ on $\overline{\mathcal{X}}$, 
\begin{align}
\E\left(\prod_{u\in \mathbb G_n} F(Box(u)) \, \vert \, \overline{\mathcal F}_n\right)=\prod_{u\in \mathbb G_n} \mu_{Z(u)}(F),
\label{branchetendu}
\end{align}
where $\mu_x$ is the law of $Box(u)$ when the trait of $u$ is $Z(u)=x$ : 
$$\mu_x(F)=\E_{(\varnothing , x)}(F(Box(\varnothing)).$$
A typical example is $Box(u)=( Z(u), \Theta_{Z(u)}(u))$ where we pair the trait of $u$ and the ones of its offspring. We will use this framework in  continuous-time by plugging the evolution of individual $u$  during the time interval $[n\delta, (n+1)\delta)$ into $Box(u)$, see forthcoming Section~\ref{sec:continuous-time-and-examples}. 
\begin{Prop} \label{extensionpsL1}
   Let $F : \overline{\mathcal X} \rightarrow \R_+$ measurable {be such} that  
   $$\parallel  F\parallel_{\overline{\mathcal B}(h)}=\sup_{(x,w)\in \overline{\mathcal X}}\frac{F(x,w)}{h(x)} <\infty.$$ Then
   $$\lim_{n\rightarrow \infty} \lambda^{-n} \sum_{u \in \mathbb G_n} F(Box(u))=W \overline{\gamma} \quad \text{a.s. and in } L^1, $$
   where $\overline{\gamma}(F)= \gamma(\mu_.(F))= \int_{\mathcal X} \gamma(dx) \mu_x(F)$.
\end{Prop}
In this result, we work with test functions $F$ strongly dominated by $h$. It will be enough for our purpose but could be relaxed by truncation arguments.
\begin{proof} We just give the main lines based on the proof of Theorem  \ref{th:main-discret}. We consider the extended empirical measure  defined for $n\geq 0$ by
$$\overline{Z}_n=\sum_{u \in \mathbb G_n} \delta_{Box(u)}, \quad \overline{X}^{(r)}_n= \frac{\overline{Z}_{nr}}{\lambda^{nr}}. $$
Our assumption on $F$ ensures that 
$$\overline{Z}_n(F)\leq \parallel F\parallel_{\overline{\mathcal B}(h)}  Z_n(V).$$ Similarly we define,
$$\overline{Z}^{(u)}_p= \sum_{ v  \in \mathcal U : uv\in \mathbb G_{n+p}} \delta_{Box(uv)}$$
and 
\begin{align*} 
\overline{\Delta}_{n+1}^{(r)}(F)&=\overline{X}_{n+1}^{(r)}(f)-\E\left(\overline{X}_{n+1}^{(r)}(F) \, \vert \, \overline{\mathcal F}_n^{(r)}\right) \\
&=\lambda^{-nr}\sum_{u\in \mathbb G_{nr}}
\left\{\lambda^{-r}\overline{Z}^{(u)}_r(F)-\E_{\delta_{Z(u)}}(\overline{X}^{(r)}(F))\right\} 
=\overline{A}_{n+1}^{(r)}(F)+\overline{B}_{n+1}^{(r)}(F),
 \end{align*}
where
 \begin{align*} 
\overline{A}_{n+1}^{(r)}(F)&=\lambda^{-nr}\sum_{u\in \mathbb G_{nr}} \left\{ \lambda^{-r}\overline{Z}^{(u)}_r(F)\mathbf{1}_{Z^{(u)}_r(V)\leq \lambda^{nr}} -\E_{\delta_{Z(u)}}(\overline{X}^{(r)}(F)\mathbf{1}_{Z_r(V)\leq \lambda^{nr}}) \right\}, \\
\overline{B}_{n+1}^{(r)}(F)&=\lambda^{-nr}\sum_{u\in \mathbb G_{nr}} \left\{\lambda^{-r}\overline{Z}^{(u)}_r(F)\mathbf{1}_{Z^{(u)}_r(V)>\lambda^{nr}} -\E_{\delta_{Z(u)}}(\overline{X}^{(r)}(F)\mathbf{1}_{Z_r(V)>\lambda^{nr}})\right\}.
 \end{align*}
The proofs can be achieved following similar arguments as above, but relying now on the filtration $\overline{\mathcal F}$. Indeed, we observe that 
$$\bar{Z}_n(F)\leq \parallel F\parallel_{\overline{\mathcal B}(h)}  Z_n(V),$$
and that $S_n(f)$ for $f\in \mtcB(V)$ is now replaced by 
 $\overline{S}_n(F)$ for $F\in \overline{\mtcB}(V)$
defined by
$$\overline{S}_nF(x)=\E_{\delta_x}\left(\sum_{u\in \mathbb G_n} F(Box(u))\right)=\E_{\delta_x}\left(\sum_{u\in \mathbb G_n} \mu_{Z(u)}(F)\right)=S_n(\mu_.(F))(x),$$
since  Assumption  \eqref{branchetendu} ensures that
$$\E(F(Box(u)) \vert \overline{\mathcal F_n})=\E(F(Box(u)) \vert Z(u))=\E_{Z(u)}(F(Box(\varnothing)))=\mu_{Z(u)}(F).$$
Note also  that we can define  $\overline{T}^{n}= \lambda^{-n}\overline{S}^{n}-\overline{\gamma} h$ and still have 
$\gamma(\overline{T}^{n}F)=0$. Moreover  
$$ \parallel \mu_.(F)   \parallel_{{\mathcal B}({h})}
\leq \parallel F\parallel_{\overline{\mathcal B}(h)}<\infty$$
 by assumption   and recall that $h\in \mathcal B(V)$ and  $V\leq V^{\star}$, so  $\parallel \mu_.(F)   \parallel_{{\mathcal B}({V}^{\star})} \leq \parallel \mu_.(F)   \parallel_{{\mathcal B}(V)}<\infty$. Now following the first lines of the proof of Lemma \ref{dec},
 we obtain the following decomposition
\begin{align*}
\overline{X}_{n+1}^{(r)}(F)=X^{(r)}_n(T_r(\mu_\cdot F))+\gamma(\mu_\cdot F) \, X^{(r)}_n(h)+\overline{\Delta}_{n+1}^{(r)}(F).
\end{align*}
Adding that $\mathcal B(h)$ is {stable under} $T_r$, we get that the two first terms of this decomposition converge a.s. and in $L^1$ thanks to Theorem~\ref{th:main-discret}

We then can show that $\overline{\Delta}_{n+1}^{(r)}(F)$ also converges to $0$ a.s.~and in $L^1$ by following the previous proof and observing simply that 
\begin{align*}
    \vert \overline{B}_n^{(r)}(F) \vert &\leq \lVert F \rVert_{\overline{\mathcal{B}}(V)} B_n^{(r)+},
\end{align*}
and
 \begin{align*}
\sum_{i=1}^\infty  \, \E\left( \overline{A}_{i}^{(r)}(F)^2 \,  \vert \, \mathcal{F}^{(r)}_{i-1} \right)  & \leq \,  \lVert F\rVert_{\overline{\mtcB}(V)}^2 \,\sum_{i=0}^{\infty}  \lambda^{-ir} X_i^{(r)}(V^{(r)}_i).
 \end{align*}
 The remaining steps follow arguments of the proof of Theorem \ref{th:main-discret}.
 \end{proof}

\section{Convergence and applications in continuous-time}\label{sec:continuous-time-and-examples}
\subsection{Main result}\label{sec:gen-cv-result}

In this section, we further assume that $\mathcal{X}$ is a separable metric space and consider a continuous-time, measure-valued, càdlàg Markov branching process $(Z_t)_{t\geq 0}$ on this space.  We refer to  forthcoming Section \ref{se:lgolfacile} for existence and details. In particular, $Z$ is constructed as finite punctual measure on $\mathcal X$ endowed with the narrow topology.\\

For any $x\in \mathcal X$, we consider its first moment semigroup  
$$ S_t f(x) = \mathbb{E}_{\delta_x}(Z_t(f)).$$
As in the previous section, it will be well defined for functions $f \in \mathcal B(V^{\star})$ where  $V^{\star}$ is non-negative and for any $t>0$, $S_tV^{\star} \in \mathcal B(V^{\star}).$

\begin{assumption}
\label{ass:cv-continu}
There exists a positive triplet $(\gamma,h,\lambda)$ of eigenelements
    such that ${\lambda >1}$, $\gamma$ is a probability on $\mathcal X$, $h : \mathcal X\rightarrow (0,\infty)$ is measurable and lower semi-continuous $\gamma$-almost everywhere and  for $t\geq 0$,
\begin{equation}
\label{eigencont}
 \gamma S_t ={\lambda^t} \gamma, \quad S_t h={\lambda^t} h, \quad  \gamma(h)=1.
 \end{equation}
 \end{assumption}

We now show the continuous-time counterpart of Theorem~\ref{th:main-discret}, where we write $W$ the limit of martingale $(\lambda^{-t}Z_t(h))_{t\geq 0}$.
\begin{theorem}\label{th:main-cont}  Under Assumption \ref{ass:cv-continu}, we further assume that there exist two positive measurable functions $V$ and  $V^{\star}$ on $\mathcal X$  such that $h\in \mtcB(V)$ and $V\in  \mtcB(V^{\star})$ and 
$\gamma(V^{\star})<\infty,$ 
and a decreasing function $a$ on $\R_+$ such that for any $t\geq 0$,
\begin{align}
\label{eq:ergo}
    \sup_{|f|\leq V^{\star}} \left| {\lambda^{-t}} S_t f(x)  - h(x) \gamma(f) \right| \leq a(t) V^{\star}(x), 
\qquad \int_1^\infty \frac{a(s)}{s} ds<\infty.
\end{align}
Assume also  that for any $t>0,$  
\begin{equation}
\label{eq:xlogx-cont}
    \sup_{x \in \mathcal X} \frac{\E_{\delta_x}(Z_{t}(V)\log^\star Z_{t}(V))}{V^{\star}(x)} <\infty.
\end{equation} 
Then, for any initial value $x\in\mathcal X$, {the limit of the martingale} $W$ satisfies  
\begin{align*}
\E_{\delta_{x}}(W)=h(x)\quad \text{and } \quad \lim_{t\rightarrow \infty}  {\lambda^{-t}} Z_t(h)= W \quad   \mathbb P_{\delta_x} \, \text{a.s. and in } L^1.
\end{align*}
Moreover, the following convergence holds  for any  bounded continuous $\gamma$-a.e.\ function $f$,
\begin{align*}
\lim_{t\rightarrow \infty}    {\lambda^{-t}} Z_t(fh)=\gamma(fh) \, W\quad \mathbb P_{\delta_x} \,  \text{a.s. and in } L^1.
\end{align*}
\end{theorem}

We start by proving a.s.~convergence and define
for any $t\geq 0,$
$$X_t(f):={\lambda^{-t}} Z_t(f).$$ 
\begin{lemma}\label{lem:conttmp}
Let $f \in \mtcB^+(h)$ { be $\gamma$ a.e. lower semi-continuous} and  {$A\subset\mathcal{X}$ be measurable} such that $\gamma(\partial A)=0$. Then
\begin{align*}
        \liminf_{t\to \infty}X_t( f\mathbf{1}_A) \geq \gamma(f \mathbf{1}_A) W \quad \text{a.s}.
\end{align*}
\end{lemma}

\begin{proof}
For any $\epsilon>0, \  x\in \mtcX,  \ A\subset \mtcX$ and $f\in \mtcB^+(h)$, following \cite{asmussen1976strong}, we introduce the following subset of $A$
$$
A_f^\epsilon(x)= \left\{ y \in A \ \big \vert \  f(y) > \frac{1}{1+\epsilon} f(x) \right\}.$$
Next, for any $t\geq 0,$ we write $t = n\delta +s$ with $n\in \N, \delta>0$ and $ s\in [0, \delta)$. We denote respectively by $\mathcal{U}_{n\delta}$ the set of individuals alive at time $n\delta$, and $\mtcU_{n\delta+s}(u)$ the set of individuals alive at $n\delta+s$ issued from individual $u$ alive at $n\delta.$ By the branching property and the definition of subsets $A_f^{\epsilon}(\cdot),$ we have 
\begin{align}
    X_t(f\mathbf{1}_A) 
      &=\lambda^{-n\delta- s}\sum_{u \in \mtcU_{n\delta}} \sum_{v\in \mtcU_{n\delta+s}(u)} f\mathbf{1}_{A}(Z_{n\delta + s}(v))
      \geq \lambda^{-n\delta-\delta} \sum_{u \in \mtcU_{n\delta}} F^{\epsilon}(Box(u)),
      \label{minorcont}
\end{align}
where we define 
$$F^{\epsilon}(Box(u))= (1+\epsilon)^{-1} f(Z(u)) \mathbf{1}_{\mathcal{A}^{\delta,\epsilon}(u)}$$
and
$$\mathcal A^{\delta,\epsilon}(u)=
\left\{ \mtcU_{(n+1)\delta}(u)\ne\emptyset\right\} \cap  \left\{\forall w\in [0,\delta), \, \forall v\in \mtcU_{n\delta +w} (u),  Z_{n\delta +w}(v)  \in A_f^\epsilon(Z(u))  \right\}.$$
Now we can use the a.s. limit of Proposition 
\ref{extensionpsL1} for the right hand side.
We obtain for $\delta,\epsilon$ fixed,
\begin{align*}
\liminf_{t\to \infty}   X_t(f \mathbf{1}_A) 
&\geq \lambda^{-\delta}(1+\epsilon)^{-1}\gamma(f\cdot\xi^{\delta, \epsilon}) W \, \  \text{ a.s.},
\end{align*}
where 
$$\xi^{\delta, \epsilon}(x)=\mathbb  P_{\delta_x}(\mathcal A^{\delta,\epsilon}\varnothing)).$$


Since $f\in \mathcal B^+(h)$ and $\gamma(h)<\infty$, we have
$\gamma(f)<\infty$. We let $\delta\downarrow0$ and then
$\epsilon\downarrow0$. By Fatou's lemma, it is enough to check that,
for every fixed $\epsilon>0$ and for $\gamma$-almost all
$x\in\mathcal X$,
\begin{align*}
\liminf_{\delta\downarrow0} f(x)\xi^{\delta,\epsilon}(x)
\geq f(x)\mathbf 1_A(x).
\end{align*}

For that purpose, the case $x\in A^c$ is immediate, since the
right-hand side is zero. Since $\gamma(\partial A)=0$ and $f$ is
$\gamma$ a.e.~lower semicontinuous, it remains to consider the case
when $x$ belongs to the interior of $A$ and to the set of lower
semicontinuity of $f$. When $f(x)=0$, we still have, for all
$\delta,\epsilon$,
\(
f(x)\xi^{\delta,\epsilon}(x)=0=f(x)\mathbf 1_A(x),
\)
and the desired inequality is obvious. Therefore, we can focus on the
case when $f(x)>0$. Fix $\epsilon>0$.
Using that $x$ belongs to the interior  of $A$ and is a point of lower semi-continuity of $f$,  there exists $r_x>0$ such that
\[
\mathcal B(x,r_x)\subset A_f^\epsilon(x).
\]
We observe that
\[
\mathcal A^{\delta,\epsilon}(\varnothing)^c
\subseteq
\{\mathcal U_\delta=\emptyset\}
\;\cup\;
\Bigl\{
\exists\, w\in[0,\delta):\exists\, v\in\mathcal U_w(\varnothing)
\text{ such that }
Z_w(v)\notin\mathcal B(x,r_x)
\Bigr\}.
\]
Since $Z_0=\delta_x$ and the population process $(Z_t)_{t\ge0}$ is right-continuous for the narrow topology on finite measures, we have $Z_t\to\delta_x$ almost surely as $t\downarrow0$. In particular,
\[
\mathbb P_{\delta_x}(\mathcal U_\delta=\emptyset)\to0
\quad\text{ and }  \quad 
\mathbb P_{\delta_x}\Bigl(
\exists\, w\in[0,\delta):
Z_w(\mathcal B(x,r_x)^c)>0
\Bigr)
\to0
\quad\text{as }\delta\downarrow0.
\]
Hence for any $\varepsilon>0$,
\[
\mathbb P_{\delta_x}(\mathcal A^{\delta,\epsilon}(\varnothing)^c)\to0, \qquad \xi^{\delta,\epsilon}(x)\to1
\quad\text{as }\delta\downarrow0,
\]
for any  $x\in\mathring A$ with $f(x)>0$ and $f$ lower semi-continuous in  $x$. It ends the proof.
\end{proof}

\begin{proof}[Proof of a.s.\ convergence in Theorem~\ref{th:main-cont}]
We work on the event $\{W>0\}$, since otherwise the limit is trivial.
Let $A\subset\mathcal X$ be measurable with $\gamma(\partial A)=0$.
Applying Lemma~\ref{lem:conttmp} to both $A$ and $A^c$, we have on one hand,
\begin{align*}
    \liminf_{t\to \infty}  X_t(h\mathbf{1}_A) \geq \gamma(h \mathbf{1}_A) W \qquad \text{a.s.},
\end{align*}
and on the other hand,
\begin{align*}
\limsup_{t\to \infty} X_t(h\mathbf{1}_A)
&= \limsup_{t\to \infty} X_t(h) -X_t(h\mathbf{1}_{A^c})\\
&\leq \limsup_{t\to \infty}  X_t(h) -\liminf_{t\to \infty} X_t(h\mathbf{1}_{A^c})\leq 
\gamma(h)W-\gamma(h{\bf 1}_{A^c})W=   \gamma(h\mathbf{1}_A)W,
\end{align*}
 which then ensures
$$\lim_{t\to \infty} X_t(h\mathbf{1}_A) = \gamma(h \mathbf{1}_A) W \qquad  \text{a.s.}$$
The family of random measures
$(\mu_t)_{t\geq 0}$ defined for $f$ {bounded} continuous $\gamma$-a.e.\ by 
$$
\mu_t(f) :=
\frac{X_t(fh)}{X_t(h)},
$$
thus verifies that $\lim_{t \to \infty} \mu_t(A)= \gamma^h(A) : = \int_A h d\gamma $ a.s.~for each continuity set of $\gamma^h$.
Using \cite[Theorem 2.3]{billingsley2013convergence}, we deduce that $\mu_t$  converges weakly towards $\gamma^h$.
Since $\mathcal X$ is a separable metric space, convergence on continuity sets implies weak convergence (see \cite[Theorem~2.3]{billingsley2013convergence}). Therefore,
\[
\mu_t \Rightarrow \gamma^h
\qquad\text{a.s.}
\]
If we denote by $D_f$ the set of discontinuities of $f$, by assumption, $\gamma(D_f)=0$, hence also $\gamma^h(D_f)=0$.
By the Portmanteau theorem (see \cite[Section~2]{billingsley2013convergence}),
\[
\int_{\mathcal X} f\, d\mu_t
\longrightarrow
\int_{\mathcal X}  f\, d\gamma^h
=
\gamma(fh)
\qquad\text{a.s.}
\]
Since $X_t(h)\to W$ a.s., we obtain
\[
X_t(fh)
=
X_t(h)\int_{\mathcal X}  f\,d\mu_t
\xrightarrow[t\to\infty]{}
\gamma(fh)\,W
\qquad\text{a.s.}
\]
This ends the proof of a.s. convergence. 
\end{proof}

We now prove $L^1$ convergence of the renormalized empirical measure by considering its positive and negative parts as follows. Recall that for any real-valued process or quantity $X:$
\begin{align*}
    X_{+}:=\max(X;0);\qquad
    X_{-}:=-\min(X;0).
\end{align*}
\begin{proof}[Proof of  $L^1$ convergence in Theorem~\ref{th:main-cont}]
We can invoke Scheff\'e's lemma together with convergence of the semigroup. We can also adapt easily the proof of the a.s. convergence as follows. 
\begin{enumerate}
\item First, observing that $\gamma(\partial \mathcal{X})=0$, we adapt the proof of Lemma \ref{lem:conttmp} and control the negative part of $X_t( f) - \gamma(f) W$ in $L^1$. We prove that for any $f\in \mathcal B^+(h)$ continuous $\gamma$-a.e.,  we have
\begin{align*}
        \lim_{t\to \infty}\E_{\delta_x}\left(\left(X_t( f) - \gamma(f) W\right)_-\right)=0.
\end{align*}
Indeed, using the lower bound \eqref{minorcont} and the triangular inequality yields
\begin{align*}
&\E_{\delta_x}\left(\left(X_{t} ( f) - \gamma(f) W\right)_-\right)\\
    & \qquad \qquad \leq
    \lambda^{-\delta} (1+\epsilon)^{-1} \E_{\delta_x}\left(\left(\lambda^{-n\delta}\sum_{u \in \mathcal{U}_{n\delta}}  \, F^{\epsilon} \, (Box(u))- \gamma(f \cdot\xi^{\delta,\epsilon}) W\right)_-\right)\\
    &\qquad \qquad  \qquad +\left(\gamma(f \cdot\xi^{\delta,\epsilon})-\gamma(f \mathbf{1}_A)\right)_- \E_{\delta_x}(W).
\end{align*}
Now we can use the $L^1$ limit of Proposition 
\ref{extensionpsL1} to make the first term of the right hand side go to zero.
We conclude by letting $\delta$ and then $\epsilon$ go to zero, so that $\gamma(f \cdot\xi^{\delta,\epsilon})$ goes to $\gamma(f)$.  Recalling the arguments at end of  the proof of Lemma \ref{lem:conttmp}, it only requires that $f$ is $\gamma$ a.e. lower semi continuous.
\item By Assumption \ref{eq:ergo},
$$\lim_{t\to \infty}\E_{\delta_x}(X_t(f))= \lim_{t\to \infty} \lambda^{-t} S_tf(x)=h(x)\gamma(f)$$ for any $f\in \mathcal{B}^+(h)$. 
\item Finally, taking expectation in the following expression   
\begin{align*}
  \left(X_t(f)-\gamma(f)W \right)_+=\left(X_t(f)-\gamma(f)W \right) -\left(X_t(f)-\gamma(f) W \right)_- 
\end{align*}
shows that its left-hand side converges to $0$ in $L^1$. 
We conclude by decomposing $X_t(f)-\gamma(f)W$  with its positive and negative parts.
\end{enumerate}
This completes the proof of the $L^1$ convergence of $X_t(f)$ for any $f\in \mathcal{B}^+(h)$ which is $\gamma$-a.e.\ lower semi-continuous.  .
\end{proof}

\subsection{Construction and preliminaries  for applications}
\label{se:lgolfacile}

Let us detail here a general way to prove both the well-posedness of the branching process (i.e.~non-explosion of the dynamics between branching events and non-explosion of the number of individuals) and the $L\log L$ condition based on infinitesimal drift conditions.

To that end, we place ourselves within a similar and general framework to \cite{cloez2017limit, marguet19}: between branching events, individuals possess a trait which evolves according to some Markovian dynamics and, depending on this trait, they branch out giving birth to a random number of descendants with new traits.

More precisely, let $\mathcal{X}$ be a locally compact and separable metric space (with its Borel $\sigma$-field). Consider a family of increasing (for the inclusion order) open sets $(O_n)$ of $\mathcal{X}$, satisfying $$\bigcup\limits_{n\geq 0} O_n=\mathcal{X}.$$ 

\paragraph{The trait dynamics.}
\label{se:trait-dyn}
Let $(Y_t)_{t\geq 0}$ be a time-homogeneous Markov Borel right process on $\mathcal{X}$. The latter will model the underlying dynamics between branching events. Let $\zeta \in \mathbb{R} \cup \{+\infty \}$ be the almost-sure limit, when $m \in \N$ tends to infinity, of the sequence of hitting times $T_m$ of $O_m^c$. If $\zeta = +\infty$ then we say that the process $(Y_t)_{t\geq 0}$ is non-explosive (or regular). Otherwise, we consider a particular abstract cemetery point $\Delta \notin \mathcal{X}$, and define $Y_t =\Delta$ for any $t\geq \zeta$.  Similarly to \cite{meyn1993stability} and references therein, we assume that the killed process $\{ Y_t: 0\leq t <\zeta\} $ is a Borel right process.
In any case, let $(Y^m_t)_{t\geq 0}$ be the process defined by $Y^m_t=Y_t$ for $t<T_m$ and $Y_t^m= \Delta$, for $t\geq T_m$. As stated in \cite{meyn1993stability}, this truncated process is shown in \cite[Theorem 12.23]{sharpe1988general} to be a non-explosive right process once $Y_t$ is assumed to be a non-explosive right process.

Let $(G,\mathcal{D}(G))$ be the extended generator of $(Y_t)_{t\geq 0}$. More precisely, $\mathcal{D}(G)$ is the set of measurable functions $f:\mathcal{X} \to \mathbb{R}$ for which  there exists a measurable function $\varphi:\mathcal{X} \to \mathbb{R}$ such that
$$
\mathbb{E}_x\left[|f(Y_t)| \right] < + \infty, \quad \int_0^t \mathbb{E}_x\left[|\varphi(Y_s)| \right] ds < + \infty,
$$
and
$$
\mathbb{E}_x\left[f(Y_t) \right] =f(x) + \int_0^t \mathbb{E}_x\left[\varphi(Y_s) \right] ds.
$$
In this case we write $\varphi=Gf$. In the last expressions, $\mathbb{E}_x$ denotes, as usual, the expectation conditioned on $Y_0=x$. This definition ensures martingale properties;
see \cite{meyn1993stability} or \cite[Chapter 1, Section 5]{ethier2009markov} for details. Similarly, we write $(G_m, \mathcal D(G_m))$ for the extended generator of $(Y^m_t)_{t\geq 0}$.



We will treat the following examples in the forthcoming applications.
We will first consider branching diffusion models, where the trait dynamic is a diffusion. In this case, $\mathcal{X}=D$ is an  open connected subset of $\R^d$,  $d\geq 1$,  $O_n= \{ x\in D \ | \ \textsc{d}(x,\partial D) > 1/n \}$ where we write $\textsc{d}$ for  the Euclidian distance in $\R^d$ and the associated distance  of a point to  a set, $\mathcal{D}(G)$ will be the set of $C^2$ functions (bounded with bounded derivative) and

$$
G_n f(x) = \sum_{i=1}^d b_i(x) \partial_{x_i} f(x) + \frac{1}{2} \sum_{i=1}^d \sum_{j=1}^d (\sigma \cdot \sigma^T)_{i,j}(x) \partial_{x_i, x_j} f(x), 
$$
for $b,\sigma$ described hereafter. We will assume $x\in O_n$ and $f\in \mathcal{D}(G)$.  \\
As a second example, we will consider 
the house-of-cards model, without any motion for the traits (i.e.~jumps will occur at branching events). In this case,  $\mathcal{X}=[0,1]$, $O_n=[0,1]$ for all $n\geq 1$, $$Gf=0$$  and $\mathcal{D}(G)$ is the set of bounded functions.\\
Finally, we will apply our results to growth-fragmentation models, where the trait follows the deterministic ODE $\dot y_t=g(y_t)$.  There, $\mathcal{X}=\mathbb{R}_+$, $O_n=(1/n,n)$ for all $n\geq 1$, 
$$
G_n f= g(x) f'(x)$$ and $\mathcal{D}(G)$ is the set of $C^1$ functions.


\paragraph{The branching mechanism.}
\label{se:banchingexe}
Instead of describing the entire population, let us describe here how the first generation of individuals is produced from an initial individual with trait $x\in \mathcal{X}$. The rest of the dynamics is then produced iteratively : each offspring will evolve independently similarly to the initial individual and, conditionally on its trait at birth, independently of the initial individual.

Let us begin by defining the branching time. Let $B$ be a locally bounded function  on $\mathcal{X}$ representing the branching rate.

Let $(Y_t)_{t\geq 0}$ be defined as in Section~\ref{se:trait-dyn}, with $Y_0=x$, and $E$ be an exponentially distributed random variable, with mean $1$. The first branching time $\beta_\varnothing$ is defined as follows : if for every $m\in \mathbb{N},$
$$
\int_0^{T_m} B(Y_s) ds < E,
$$
then we set $\beta_\varnothing=\zeta$ and, else we set
$$
\beta_\varnothing= \inf\left\{ t\geq 0 \ | \ \int_0^{t} B(Y_s) ds \geq E \right\}.
$$

We then set $X_\varnothing(t) = Y(t)$ and $Z(t) = \delta_{X_\varnothing(t)}$ for every $t<\beta_\varnothing$.

Let us now model the offspring. At time $\beta_\varnothing$, the first individual is removed and replaced by a random number of new individuals. More precisely, let $x \mapsto (q_k(x))_{k\in \N}$ be a measurable function from $\mathcal{X}\cup \{\partial \}$ to the set of  probabilities over $\mathbb{N}$ and, for any $k\in \mathbb{N}$,  let $Q^k : x \mapsto Q^k(x, d x_1, \dots, dx_k) $ be a measurable function from $\mathcal{X} \cup \{\partial\} $ to probabilities on $(\mathcal{X}\cup \{\partial\} )^k$.

We can now define $Z(\beta_\varnothing)$ on the event $\beta_\varnothing<\infty$. Let $\nu_\varnothing$ be a random variable distributed such that 
$$
\forall k \geq 0, \ {\bf 1}_{ \beta_\varnothing<\infty}\mathbb{P}\left( \nu_\varnothing =k \ | \ (X_\varnothing(t))_{t< \beta_\varnothing} \right) = {\bf 1}_{ \beta_\varnothing<\infty}\, q_k(X_\varnothing(\beta_\varnothing)).
$$
The variable $X_\varnothing(\beta_\varnothing)$ corresponds to $Y(\beta_\varnothing)$ and is equal to $\partial$ in the case $\beta_\varnothing=\zeta$. Finally, let $(X_1(\beta_\varnothing),...,X_{\nu_\varnothing}(\beta_\varnothing))$ be  a random vector whose law, conditionally on $\{(X_\varnothing(t))_{t< \beta_\varnothing}, \nu_\varnothing \}$, is given by $Q^{\nu_\varnothing} (X_\varnothing(\beta_\varnothing))$. We set $Z(\beta_\varnothing)=\sum_{k=1}^{\nu_\varnothing} \delta_{X_k(\beta_\varnothing)}$, where $Z(\beta_\varnothing)=0$ if $(\nu_\varnothing)=0$.

Finally, starting from the stopping time $\beta_\varnothing$, the dynamics of the measure $Z$ is described by the sum of particles $\nu_\varnothing$ evolving and branching independently like the first.

This dynamics is well defined as long as the number of jumps is not infinite in finite time; we will describe sufficient conditions for this to hold in the next section.


\paragraph{Non-explosion and generator.}
From {now on}, we fix some measurable function $V:\mathcal{X} \to [1,\infty)$ which belongs to $\mathcal{D}(G)$. For any $m\geq 0$, we set $$\mathcal{O}_m=\{ x\in \mathcal X : V(x)\leq m\} $$ and will assume that for every $m\geq 0$,
\begin{equation}
\label{eq:Blocbound}
\sup_{x \in \mathcal{O}_m} B(x) < + \infty.
\end{equation}
We refer to \cite{hairer2011yet} for examples. 
 In the case where particles never reach
 $\partial$, we will consider $O_n=\mathcal{O}_n$.

Under Assumption~\eqref{eq:Blocbound}, as long as the process contains a bounded number of particles belonging to one of the sets $\mathcal{O}_m$, then the number of jumps can be bounded by coupling it with {that} of a mono-type branching process.

Here, {proving non-explosion  consists of}
 showing that
$$ \lim_{n \to \infty} \mathcal{T}_n=+\infty \quad \text{a.s.}, \quad \text{where }
\mathcal{T}_n = \inf\{ t\geq 0 \ | \ Z_t(\mathbf{1})\geq n \text{ or } Z_t(\mathbf{1}_{\mathcal{O}_n^c})>0 \},
$$
by using Lyapunov functions for the generator \cite{meyn1993stability}. This will also enable to exhibit sufficient conditions for the   $L \log L$ moment condition.
 Let $Z^n$ be the process killed at time $\mathcal{T}_n $. Namely $Z^n(t)=Z(t)$ for $t<\mathcal{T}_n$ and  $Z^n(t)=\Delta$, i.e.~some cemetery point (as before) for $t\geq \mathcal{T}_n$. For sake of notation, we consider $\partial$ as any abstract positive measure. 
We can describe the generator of the Markov process $(Z^n_t)_{t\geq 0}$ for functions $F_f : \mu \mapsto F(\mu(f))$ on finite {point measures}, where $F$ is $C^1(\mathbb R_+,\mathbb R)$ and $f\in \mathcal{D}(G)$. Indeed, let us consider such a functional $F_f$ and let $\mu$ be a point measure (with $m\leq n$ atoms).
Between jumps, the individuals evolve independently with generator $G_n$. A simple application of the chain rule yields 
the first term in formula \eqref{eq:extgen} below. For the jump part, conditionally on $Z_{t-}^n=\mu$, an individual at trait $x$ branches at rate $B(x)$ and is replaced by $k$ offspring with traits distributed according to $Q^k(x,\cdot)$. Hence the generator contribution is obtained by integrating 
\[
F_f\!\left(\mu-\delta_x+\sum_{i=1}^k\delta_{x_i}\right)-F_f(\mu)
\]
against the branching rate and offspring law. Evaluating this expression yields exactly the second and third terms of \eqref{eq:extgen} below, including the truncation contribution when the population size exceeds $n$. This leads to
\begin{align}
&A_n F_f(\mu) = \mu(G_n f) F'(\mu(f))\nonumber \\
&+ \int_{\mathcal X} B(x) \sum_{k=0}^{n-m} q_k(x) \left(\int_{\mathcal X^k} F\left(\mu(f) -f(x) + \sum_{i=1}^k f(x_i) \right) Q^k(x,dx_1,...,dx_k) - F(\mu(f)) \right) \mu(dx)\nonumber\\
&+ \int_{\mathcal X} B(x) \sum_{k\geq n-m+1} q_k(x) \left( F\left(\Delta(f) \right)  - F(\mu(f)) \right) \mu(dx).
\label{eq:extgen}
\end{align}
We refer to e.g.  \cite{roelly1990construction}for details of the proof.\\
In particular, letting $f=V, F = \textsc{Id}$ and $\mu = \delta_x$ (hence $m=1$), we can rewrite \eqref{eq:extgen}  
\begin{align*}
    A_n\textsc{Id}_V (\delta_x) &=G_n V(x) +  B(x) \sum_{k=0}^{n - 1} q_k(x) \int_{\mathcal X^k}\left(-V(x) + \sum_{i=1}^k V(x_i) \right) Q^k(x,dx_1,...,dx_k)\\
    &+ B(x) \sum_{k \geq n}q_k(x)(V(\Delta) - V(x))\\
  &\leq G_n V(x) +  B(x) \sum_{k=0}^{n-1} q_k(x) \int_{\mathcal X^k}\left(-V(x) + \sum_{i=1}^k V(x_i) \right) Q^k(x,dx_1,...,dx_k),
\end{align*}
 where for the last inequality we used $V(\Delta) = 0 $ and $B(x),q_k(x), V(x)\geq 0.$

Let us now give {sufficient conditions for} non-explosion and to satisfy the $L \log L$ condition. 
We write 
 $\mathcal{P}_n$ for the set of {point measures} $\mu$ such that $\mu(O_n^c)=0$ and $\mu(\mathbf{1}) \leq n.$
\begin{lemma}
\label{lem:nonexplosion}
i) If there exists a 
function $V$ such that,
$$
\sup_{n\geq 0} \sup_{x\in O_n} \frac{A_n\textsc{Id}_V (\delta_x)}{V(x)} <+\infty,
$$
then there is no explosion : $\lim_{n \to \infty} \mathcal{T}_n=+\infty$. \\

ii) If moreover, a function $F$ satisfies
$\inf_{y>0} F(y)/(y\log^\star(y))>0$ and
\begin{align*}
&\sup_{n\geq 0} \sup_{\mu \in \mathcal{P}_n} \frac{A_n F_V(\mu)}{F_V(\mu)} <+ \infty,
\end{align*} then for any $t\geq 0$,
$$\sup_{x\in \mathcal X} \frac{\E_x( Z_t(V)\log^{\star}(Z_t(V))}{F(V(x))}<\infty .$$
\end{lemma}
Observe that Part ii) yields   $L \log L$ condition \eqref{eq:xlogx-cont}  with $V^\star(x)=F_V(\delta_x)= F(V(x))$. We also note that non explosion and the fact that the trait process $Y$ is Borel right guarantees that the branching process is c\`adl\`ag for the narrow topology.

\begin{proof}
If we consider $F :x \mapsto x$ (with $F_V(\Delta)=0$) then, the function $F_V$ verifies \cite[CD0]{meyn1993stability}. Consequently \cite[Theorem 2.1 (i)]{meyn1993stability} gives the non-explosion i). Under the second assumption $ii)$, the function $F_V$ verifies \cite[CD0]{meyn1993stability},  and \cite[Theorem 2.1 (iii)]{meyn1993stability} gives the result.
\end{proof}

In particular, when $V\equiv 1$, $\lim_{n \to \infty} \mathcal{T}_n=+\infty$ a.s. as soon as
$$
    \sup_{x\in \mathcal{X}} \left(B(x) \sum_{k\geq 0} q_k(x)(k-1)\right) <+\infty,
$$
because $G_n \mathbf{1} \leq 0$. Moreover, for the same reason, {by assuming $\mu = \delta_x, V=1, F = \textsc{Id}$} the $L \log L$ condition \eqref{eq:xlogx-cont} holds with $V^\star\equiv 1$ as soon as 

\begin{equation}
\label{eq:lgolfacile}
\sup_{x\in \mathcal{X}} \left( B(x) \sum_{k\geq 1} k \log(k) q_k(x) \right) <+ \infty.
\end{equation}
This condition is thus sufficient to construct the branching process for any time and apply Theorem \ref{th:main-cont}.

\subsection{Application to branching elliptic diffusion}\label{sec:Bdiffusion}
We  place ourselves in the framework of absorbed diffusion processes with killing as per \cite[Section 4.4]{champagnat2023general} and obtain here  conditions for the renormalized branching process to converge. 
\paragraph{Description of the branching Markov process.}
Let $\mathcal{X}=D$ be an  open connected subset of $\R^d, \ d\geq 1$. We consider a branching Markov process where each particle $u \in \mathcal{U}$ is characterised at time $t\geq 0$ by a trait $X_t^u \in D$. Between branching events, the dynamics of the trait is described by a process $(X_t)_{t\geq 0}$, solution to the following stochastic differential equation (SDE)
\begin{align}
    dX_t = b(X_t)dt+\sigma(X_t) dB_t,\label{eq:SDE-CV}
\end{align}
where $(B_t)_{t\geq 0}$ is a standard $r$-dimensional Brownian motion, $b:D\to \R^d$ and $\sigma: D\to \R^{d\times r}$ are both locally H\"{o}lder functions and $\sigma$ is locally uniformly elliptic in $D,$ i.e.
\begin{align*}
    \forall K \subset D \ \text{ compact, } \ \inf_{x\in K} \inf_{s \in \R^d \setminus \{0\}} \frac{s^T \sigma(x)\sigma^T(x)s}{\vert s \lvert ^2} >0,
\end{align*} 
and $\lvert \cdot \lvert $ is the standard Euclidean norm on $\R^d.$
The diffusion process $(X_t)_{t\geq 0}$ is assumed to be immediately absorbed at a cemetery point $\partial \notin D$ at the first exit point $\tau_{\text{exit}}$ of $D$ defined as
\begin{align*}
    \tau_{\text{exit}}=\inf \{t\geq 0: X_t \in D^c\}.
\end{align*}
Note here that \cite{chen2017law} have already studied the case of reversible processes, which ensures symmetry of the semigroup and allows reliance on other techniques.
As mentioned in \cite{champagnat2023general}, the proof of existence and the construction of such a diffusion process necessitate some work since the coefficients $b(\cdot)$ and $\sigma(\cdot)$ are only defined on the open set $D$ but not at the boundary point $\delta.$ We refer the reader to \cite[Section 12.1]{champagnat2023general} for the construction of a process $(X_t)_{t\geq 0}$ as a weak solution to \eqref{eq:SDE-CV} up to the first exit time 
\begin{align*}
    \tau_{K_k^c}:=\inf \{t\in \R_+: X_t \in K_k^c \}
\end{align*} 
of each compact subset $K_k \subset D$ defined for any $k\in \N^*$ as
\begin{align*}
    K_k:=\{x \in D: \vert x\vert \leq k \text{ and } \textsc{d}(x,D^c)\geq 1/k \},
\end{align*}
where $\textsc{d}$ is the Euclidean distance between a point and a set. 
In this case, $\tau_{\text{exit}}=\sup_{k\geq 1}\tau_{K_k^c}$. Let us now describe the branching events. Any particle $u\in \mathcal{U}$ with trait $X_t^u=x$ branches at a rate $B(x)$ and produces $k\in \N$ offspring with probability $p_k(x)$. These offspring have the same trait $x$. We denote the mean number of offspring $m(x):=\sum_{k=1}^\infty kp_k(x)$ for any $x\in D.$ Letting $V_t$ be the set of particles alive at time $t\geq 0$, the structured population is described by the branching process $  Z_t =\sum_{u\in V_t} \delta_{X_t^u}$.
We use condition \eqref{eq:lgolfacile}  
for non explosion and $L\log L$ condition: 
\begin{assumption} [related to the branching events]\label{assu:existence_diff} Assume that for every $k\geq 0$, $x\mapsto p_k(x)$ and $x\mapsto B(x)$ are continuous over $D$ and
\begin{align*} 
\sup_{x\in \mathcal{X}} B(x) \sum_{k\geq 1} k \log(k) p_k(x) <+ \infty.
\end{align*}
\end{assumption}
In particular, with this assumption,  we can define
$$
   \overline{B}:=\sup_{x \in D}\Big( B(x)(m(x)-1) \Big) <\infty.
$$

\paragraph{Asymptotic behaviour.}
 Let us introduce the measurable locally bounded function $\kappa$ on $D$ given by
\begin{align}
        \kappa(x)&:=\overline{B}-B(x)(m(x)-1) \geq 0.
\end{align}
We consider the diffusion process $(X_t)$ to be the weak solution to \eqref{eq:SDE-CV} as defined in the previous section. The corresponding process killed at rate $\kappa$ is denoted $X^\kappa$. More precisely, for an independent exponential random variable $\zeta$ with parameter 1, we set
\begin{align*}
    \tau_\partial=\tau_{\text{exit}}\wedge \inf  \left \{t\geq 0, \int_0^t \kappa(X_s)ds >\zeta \right\},
\end{align*}
and define $X^\kappa_t= X_t$ for $t\leq \tau_\partial$ and $X^\kappa_t=\partial$ for $t\geq \tau_\partial$.
 We also introduce for some $x\in D$ and open ball {$\mtcB$ such that $\overline{\mtcB} \subset D,$} the constant
\begin{align*}
    \lambda_0:={\inf} \{\ell>0, \ \text{s.t.} \ \liminf_{t\to \infty} e^{\ell t}\mathbb{P}_x(X_t \in {\mtcB})>0 \}.
\end{align*}
Here $\mathbb{P}_x$ denotes the probability conditioned on $X_0=x$. It is proven in \cite[Section 12.2]{champagnat2023general} that in the above context $\lambda_0<\infty$ and $\lambda_0$ depends neither on $x$ nor on {$\mtcB$}.
{We also denote by
\begin{align*}
    \tau_{K_k} := \inf \{ t\geq 0: \ X_t \in K_k\},
\end{align*}
the first entrance time of the process into $K_k.$}
\begin{assumption}[related to the diffusion]
\label{assu:diff-sg}
  There exists a subset $D_0 \subsetneq D$ and a time $s_1 >0$ such that 
    \begin{align*}
        \inf_{x\in D\setminus D_0} \kappa(x)>\lambda_0, \qquad \sup_{x\in D_0} \mathbb{P}_x(s_1< \tau_\partial \wedge \tau_{K_k})\longrightarrow 0, \quad \text{ as } k\to \infty.
    \end{align*} 
\end{assumption}
As specified in \cite[Remark 11]{champagnat2023general}, a simple and sufficient condition is given by $\lim_{k\to \infty} \inf_{x\in D\setminus K_k} \kappa(x)=°\infty$.
We focus on the supercritical regime when this rate is positive. 
We  have then the following result  for  branching diffusions.
\begin{theorem}
Under Assumptions~\ref{assu:existence_diff} and \ref{assu:diff-sg}, there exist  $\lambda \in \mathbb R$ and a probability measure
$\gamma$ on $\mathcal X$ and   a function $h : \mathcal X\rightarrow (0,\infty)$   measurable and lower semi-continuous  $\gamma$-almost everywhere such that for any $t\geq 0$,
\begin{equation}
\label{eigencont-bis}
 \gamma S_t ={\lambda^t} \gamma, \quad S_t h={\lambda^t} h, \quad  \gamma(h)=1.
 \end{equation}
If additionally  $\lambda>1$, then for any $x\in \mathcal X$,  the martingale satisfies  
\begin{align*}
\E_{\delta_x}(W)=h(x)\quad \text{ and } \quad \lim_{t\rightarrow \infty}  {\lambda^{-t}} Z_t(h)= W \quad  \mathbb P_{\delta_x} \text{ a.s. and in } L^1.
\end{align*}
Moreover, the following convergence holds for any bounded measurable function $f:\mathcal X\to\mathbb R$ whose set of discontinuities is $\gamma$-negligible:
\begin{align*}
\lim_{t\rightarrow \infty} \lambda^{-t} Z_t(fh)
=
\gamma(fh)\, W
\quad \mathbb P_{\delta_x}\text{-a.s. and in }L^1.
\end{align*}

\end{theorem}

\begin{proof}
Let us define the sub-Markovian semigroup $(P_t)_{t\geq 0}$ as
    \begin{align*}
        P_t=e^{-\overline{B}t}S_t,
    \end{align*}
which, by using a Feynman-Kac formula or spinal decomposition \cite{englander2004local,englander2014spatial}, can be seen to correspond to a (non-branching) killed diffusion process. We now justify that there exists a function $V$, a triplet $(\lambda_1, h,\gamma)$, with $\gamma(V)<+\infty$, $0< h\leq V$, and  $C,\rho>0$ such that for any $t\geq 0$,
\begin{align}
\label{eq:ergoPt}
    \sup_{|f|\leq V} \left| e^{-\lambda_1 t} P_t f(x)  - h(x) \gamma(f) \right| \leq C e^{-\rho t} V(x).
\end{align}
This will provide the eigenelements of \eqref{eigencont-bis} and allow us to check the assumptions of our main result.  More precisely,  
all arguments to prove \eqref{eq:ergoPt} are  contained  in \cite{champagnat2023general} and we  briefly justify why \eqref{eq:ergoPt} holds.    Assumption~\ref{assu:diff-sg} means that $X^\kappa$ verifies the assumptions of \cite[Theorem 4.5]{champagnat2023general}. From the analysis of their proofs, we deduce that the latter implies that $X^\kappa$ also verifies \cite[Assumption (F)]{champagnat2023general} with some $\psi_2\leq \psi_1$ and $\psi_1=1$. Consequently, it verifies \cite[Assumption (E)]{champagnat2023general} at discrete-times with $\varphi_1=1$. We can then observe thanks to \cite[Corollary 2.4]{champagnat2023general} that the latter result implies that \eqref{eq:ergoPt} holds at discrete-times with $V=1$. To be convinced that the result also applies in continuous-time, we can consult \cite{bansaye2022non}, given that \cite[Assumption (E)]{champagnat2023general} implies \cite[Assumption A]{bansaye2022non}, and \cite[Theorem 2.1]{bansaye2022non} implies \eqref{eq:ergoPt}. 
Recalling that Assumption \ref{assu:existence_diff} 
yields the $L\log L$ moment condition allows us to apply Theorem \ref{th:main-cont} which ends the proof,
 with $\log(\lambda):= \lambda_1 + \bar{B}$.
\end{proof}

\subsection{Application to the \textit{House of Cards} model}\label{sec:HouseOfCards}
We consider a very simple branching model where each particle has a trait $x\in \mathcal{X}= [0,1]$. Between branching events, traits remain constant. We assume that each particle, with trait $x$, branches at a continuous rate $B(x)$. At this branching event, the individual dies and gives birth to $k$ descendants with (continuous) probability $p_k(x)$ with same trait. We also assume that at rate $1$, each particle survives but gives birth to new individuals whose traits are uniformly distributed over $[0,1]$. We assume that $x\mapsto p_k(x)$ and $x\mapsto B(x)$ are continuous (and then bounded) over $[0,1]$. We further assume
$$
\sup_{x\in \mathcal{X}} B(x) \sum_{k\geq 1} k \log(k) p_k(x) <+ \infty,
$$
which gives through Lemma~\ref{lem:nonexplosion} both non-explosion and the $L\log L$ condition. In this case, the mean semigroup $(S_t)_{t\geq 0}$  associated to this dynamics is generated by
$$
\mathcal{A} f(x) = \int_0^1 f(u) du + B(x) \sum_{k\geq 0} (k-1) p_k(x) f(x).
$$
This semigroup was studied in \cite{cloez2024fast}, and references therein. In particular, it is less regular than in other contexts where a law of large numbers is generally proved. For instance, in contrast to diffusion processes,  this semigroup does not lead to an absolutely continuous measure with respect to the Lebesgue measure. However, setting $$\alpha(x)=-B(x) \sum_{k\geq 0} (k-1) p_k(x),$$
this semigroup was studied in \cite{cloez2024fast}, where $a = \alpha-\min(\alpha)$.  We can then apply our main result to obtain the long time behavior of the empirical measure.

\begin{theorem}\label{th:exa-hoc} Assume that $\alpha$ is decreasing and 
 $$\int_0^1 (\alpha(x)-\min(\alpha))^{-1} dx >1, \quad 
\sup_{x\in \mathcal{X}} B(x) \sum_{k\geq 1} k \log(k) p_k(x) <+ \infty.
$$
Then \eqref{eq:ergo} is satisfied for $V=V^{\star}=1$ and some continuous positive function $h$ on $[0,1]$ and some probability measure $\gamma$  on $[0,1]$. \\
Assuming further that $\alpha$ is non negative,  then $\lambda >1$  and for any  $x\in [0,1]$, the martingale limit $W$ satisfies   
\begin{align*}
\E_{\delta_x}(W)=h(x)\quad \text{and } \quad \lim_{t\rightarrow \infty}   \lambda^{- t} Z_t(h)= W \quad \mathbb P_{{\delta_x}} \text{ a.s. and in } L^1.
\end{align*}
Besides, the following convergence holds for any  function $f$  continuous on $[0,1]$,  
\begin{align*}
\lim_{t\rightarrow \infty}    \lambda^{-t} Z_t(fh)=\gamma(fh) \, W \quad   \mathbb P_{\delta_x}\, \text{a.s. and in } L^1.
\end{align*} 
\end{theorem}
\begin{proof} The estimate \eqref{eq:ergo} on the first moment semigroup is a consequence of 
\cite[Theorem 1.1]{cloez2024fast}. 
Similarly (but more directly) to Section~\ref{sec:Bdiffusion}, we need to add the normalization
 by $\exp(-\min(\alpha) t)$ to derive the asymptotic behavior of our first moment semigroup from this paper.
In particular, the principal eigenvalue $\lambda_1$ given by \cite[Theorem 1.1]{cloez2024fast} is positive  and  our eigenvalue $\lambda$ writes
 $$\log(\lambda)=\lambda_1+\min (\alpha)>0,$$ 
 since  $\alpha$ is non negative. Thus  the branching process is supercritical.    Also, the eigenvector given by \cite[Theorem 1.1]{cloez2024fast} is semi-explicit and inherits its regularity from the one of $\alpha$.
Recalling that  Equation~\eqref{eq:lgolfacile} implies \eqref{eq:xlogx-cont}, we can apply 
Theorem~\ref{th:main-cont}  and conclude.
\end{proof}
As shown in \cite{cloez2024fast}, when $\int_0^1 (\alpha(x)-\min(\alpha))^{-1} dx \leq 1$, convergence is no longer exponential and the limiting measure $\gamma$ may be degenerate. In certain cases, polynomial convergence can be obtained. The corresponding estimations  can differ  from  Assumption~\ref{ass:cvSt} and seem to involve  a third Lyapounov function. This leads an interesting case for future works, which may  need a technical and  delicate of the proofs given in this paper.

\subsection{Application to some growth-fragmentation models}
\label{GFMod}

Let us apply our result to a growth fragmentation process, with exponential growth and binary division. 

This example has been extensively studied in the literature, see for example \cite{bertoin2017markovian,villemonais2025quasi,mischler2016spectral} and references therein. From the point of view of the law of large numbers, this example does not fit into  previous general results as \cite{englander2010strong,harris2010strong} due to absence  of densities, reversibility or  compactness, while some interesting classes have been studied in \cite{horton2020strong,bertoin2020strong,tomavsevic2022ergodic}.

The fragments take values in $\mathcal X=(0,\infty)$,  the growth of the cell is exponential with rate $1$ and the division rate $B$ is regular and increasing : 
$ B\in C^1((0,\infty), \R_+)$. 
At branching events, a particle with trait $x$ divides into two new particles, respectively with traits $\theta x$ and $(1-\theta)x$, where $\theta$ is a random variable on $(0,1)$ with some fixed law $\vartheta$.  We assume that the fragmentation kernel satisfies
$$\vartheta(d\theta)  \geq \frac{\mathbf{1}_{[\Theta_0-\epsilon, \Theta_0]}}{c_0}d\theta ,$$
for some $\Theta_0\in (0,1), \epsilon \in [0,\Theta_0], c_0>0$.

This assumption guarantees that the underlying semigroup creates density at division, enabling us to prove convergence in total variation. This type of assumption is necessary for this type of model to avoid some pathological behavior \cite{cyclic}.

Following the notation of Section \ref{se:lgolfacile}, 
the generator of the measure-valued branching process is then given  by
\begin{align}
&A_n F_f(\mu) = A F_f(\mu) = \mu( \textsc{Id}. f') F'(\mu(f))\nonumber \\
&\qquad \qquad + \int_{(0,+\infty)} B(x) \left(\int_{0}^{1} F\left(\mu(f) -f(x) + f(\theta x) + f((1-\theta)x) \right) \vartheta(d\theta) - F(\mu(f)) \right) \mu(dx). \nonumber 
\end{align}
The first moment semigroup
is  well known and described by \cite[Theorem 3.1 (iii)]{gabriel:tel-03144625}. However, the study of the branching process for a general branching rate $B$ has, up to our knowledge, not yet been achieved. We focus here on non-explosion and $L \log L$ condition and compensations in the mechanisms at infinity. A higher branching rate implies that the empirical measure supports the compacts but increases the number of individuals.

Note here that the $L \log L$ condition does not only rely on the number of offsprings in contrast with multi-type branching processes.  Indeed, as we need to use a non constant Lyapunov function then the $L \log L$ condition is a condition on the total mass of the population.


\begin{theorem} Under the above assumptions, the growth fragmentation  process $Z$ is non explosive  and well defined on $\R_+$. Moreover there exists a unique positive eigentriplet $(\lambda, \gamma, h)$ where $h$ is continuous on $\R_+$ and $\lambda>1$  solution of \eqref{eigencont},  and $\gamma$ probability measure on $[0,1]$   and for any $x\in \mathbb R_+$,
  \begin{align*}
\E_{\delta_x}(W)=h(x)\quad \text{and } \quad \lim_{t\rightarrow \infty}   {\lambda^{-t}} Z_t(h)= W \quad  \mathbb P_{\delta_x} \text{ a.s. and in } L^1
\end{align*}
and  for any $f\in \mtcB(h)$, 
\begin{align*}
\lim_{t\rightarrow \infty}    {\lambda^{-t} }Z_t(fh)=\gamma(fh) \, W\quad  P_{\delta_x} \text{a.s. and in } L^1.
\end{align*}
\end{theorem}
\begin{proof}
Let us simultaneously prove this result and provide a few hints of extensions, in particular with respect to general growth. {Between branching events, the trait of each individual
evolves deterministically according to the ODE $\dot{x}(t):= \frac{d}{dt}x(t)=g(x(t)),$  where  $g:(0,\infty)\to(0,\infty)$ is a given $C^1$ growth rate function whose flow is well defined and unique for all $t\ge0$
under the above regularity assumptions.} 

The corresponding generator is
\begin{align}
&A_n F_f(\mu) = A F_f(\mu) = \mu( g f') F'(\mu(f))\nonumber \\
&\qquad \qquad + \int_{(0,+\infty)} B(x) \left(\int_{0}^{1} F\left(\mu(f) -f(x) + f(\theta x) + f((1-\theta)x) \right) \vartheta(d\theta) - F(\mu(f)) \right) \mu(dx). \nonumber 
\end{align}
We work with $O_n=\mathcal{O}_n=(0,n)$ and we consider a function $V$ which tends to infinity at infinity but remains bounded at 0, instead of a function tending to infinity at both boundaries of the domain.

Let us see how Lemma~\ref{lem:nonexplosion}  applies here. Considering  $F:x\mapsto x {\log^\star(x)}$,  and $V$ verifying  for all $x\in (0,+\infty)$ and 
$\theta \in (0,1)$,
$$ -V(x) + V(\theta x) + V((1-\theta)x) \leq 0
$$
leads by monotonicity of $F$ to
\begin{align*}
A F_V(\mu) 
&\leq \mu( g V') (1+ \log(\mu(V))).
\end{align*}
The drift condition of Lemma~\ref{lem:nonexplosion} is then verified as soon as $gV'\leq C V$, for some constant $C>0$.  Taking $V: x\mapsto x^p$ (or $V: x\mapsto 1+x^p$), we then capture growth rates $g$ satisfying 
$g(x) \leq C x,$ for all $x>0$ and some fixed $C>0$, which covers a large range of models of the literature, in particular our example here $g=\textsc{Id}$.   We  use  \cite[Theorem 5.3]{bansaye2022non} to prove exponential convergence of the first moment semigroup. As $G\mathbf{1}>0$, we necessarily have $\lambda>1$. The regularity of $h$ may be proven as in \cite[Lemma 4.1]{cloez2021long}. Then Assumption \eqref{eq:ergo} is satisfied and  we are then in a position to apply Theorem~\ref{th:main-cont} and obtain the results.
\end{proof}

We observe that we can study more complex growth-fragmentation models. Indeed, let us  fix $p>1$ and
$$
\alpha_p=1-\int_0^1 (\theta^p + (1-\theta)^p) \vartheta(d\theta) >0.
$$
Consider now $F :x \mapsto x^2$ and $V :x\mapsto x^p$. We have
\begin{align}
&A F_V(\mu) = 2\mu( g V') \mu(V) \nonumber 
+ \int_{(0,\infty)} B(x) \left( -2 \mu(V) \alpha_p V(x) + \alpha_p^2 V(x)^2 \right) \mu(dx).
\end{align}
The drift condition of Lemma~\ref{lem:nonexplosion} is verified as soon as

$$
\limsup_{x\to \infty} (p g(x)/x - 2 B(x)\alpha_p) <+\infty.
$$
This then enables us to ensure that the non-explosion condition holds and that the $L \log L$ condition also holds for branching processes whose growth rates are of the form $g:x\mapsto x^2$, for which the associated ODE is explosive, as soon as $\liminf_{x \to \infty}B(x)/x = + \infty$.

\begin{appendix}
  \section{Appendix: $L\log L$ moment estimates}

We prove the two inequalities relied upon in Proposition \ref{Propos-propag}.

Under Assumption \ref{ass:cvSt}, assuming that $V:\mathcal{X} \to (0,\infty)$ is a measurable function such that $V\leq V^\star$ and $V\log(V)\in \mathcal B(V^{\star})$, we derive upper-bounds in the next two lemmas for the quantity 
$$I_n= \sup_{x \in \mathcal X} \frac{\E_{\delta_x}(Z_n(V)\log^{\star} Z_n(V))}{V^\star(x)} \quad (n\geq 1).$$
\begin{lemma}
\label{Lemma-propag-i}
If $SV 
        \in \mathcal{B}(V)$, then there exists a constant $K_V\in [1,\infty)$ such that  for any $n\in \N^*$ such that $n=mq+r,$ with $(m,q) \in (\N^*)^2$  and  $r \in \mathbb{N}, 0\leq r \leq m-1$, 
\begin{align*}
\label{I_nleqI_r-I_q}
        I_n \leq c\, K_V^{n} \, \Big( I_r +I_q+1 \Big).
\end{align*}

\end{lemma}

\begin{proof}
     For any $n\in \N^*$ and $f
\in \mathcal{B}(V^\star),$ we introduce the function $x\mapsto i_n (f) (x)$ defined for any $x \in \mathcal{X}$ by
\begin{align*}
    i_n(f)(x):=\E_{\delta_x}(Z_n(f)\log^\star Z_n(f)).
\end{align*}
In the case where $f=V$, we simplify the notation and write $i_n(x):=i_n(V)(x)$.  Observe as well that for any $n\in \N^*,$ 
\begin{align*}
    I_n=\sup_{x\in \mathcal{X}} \left(\frac{i_n(x)}{V^\star(x)}\right).
\end{align*} 
Let us write $n=mq+r$ with $0\leq r \leq m-1$. Using the branching property,  

$$Z_n(V)=\sum_{u\in \mathbb{G}_{n-q}}{Z_q^{(u)}(V)}$$ 
whereby, conditionally on $\mathcal F_{n-q}$, the variables $Z_q^{(u)}(V)$ are independent for $u\in \mathbb{G}_{n-q}$. Thus we can apply {Lemma 1 of \cite{asmussen1976strong} adapted to conditional expectations} to obtain
\begin{align*}
    i_n(x)\leq \E_{\delta_x}\Big( \E(Z_n(V)\vert \mathcal{F}_{n-q})\log^\star \E(Z_n(V)\vert \mathcal{F}_{n-q})+\sum_{u\in \mathbb{G}_{n-q}}\E(Z_q^{(u)}(V)\log^\star Z_q^{(u)}(V)\vert \mathcal{F}_{n-q}) \Big).
\end{align*}
Next, observing that $\E(Z_n(V)\vert \mathcal{F}_{n-q})=Z_{n-q}(S_q V),$ 
\begin{align}
\E_{\delta_x}\Big( \E(Z_n(V)\vert \mathcal{F}_{n-q})\log^\star \E(Z_n(V)\vert \mathcal{F}_{n-q})\Big) 
&= \E_{\delta_x}(Z_{n-q}(S_qV)\log^\star Z_{n-q}(S_qV)) 
\end{align}

Adding that $SV\in \mathcal{B}(V),$  there exists a finite constant $K_1\geq 1$ such that $S_qV \leq K_1^qV$ for any $q\in \N.$ Using that $\log^{\star}(ab)\leq \log^{\star}(a)+\log^{\star} (b)$ for $a,b\geq 0$,
it leads to 
\begin{align*}
  \E_{\delta_x}(Z_{n-q}(S_qV)\log^\star Z_{n-q}(S_qV))
&\leq K_1^q   \E_{\delta_x}(Z_{n-q}(V)(\log^{\star}(Z_{n-q}(V))+q\log(K_1)))\\
&\leq K_1^q(i_{n-q}(x)+q\log(K_1)K_1^{n-q}V(x)).
\end{align*}
Next we re-write the second term 
\begin{align*}
\E_{\delta_x} \Big( \sum_{u\in \mathbb{G}_{n-q}}\E(Z_q^{(u)}(V)\log^\star Z_q^{(u)}(V)\vert \mathcal{F}_{n-q})   \Big)&=\E_{\delta_x}\Big(  \sum_{u\in \mathbb{G}_{n-q}} \E_{Z(u)} (Z_q(V)\log^\star Z_q(V)) \Big)  \\
 & = S_{n-q}i_q(x).
 \end{align*}
Observe that assuming $I_q<\infty$ at this point implies $i_q\in \mathcal{B}(V^\star)$, which in turns means that we can apply the first part of Condition \eqref{ergodnc} to the function $i_q$ and obtain 
\begin{align*}
   \E_{\delta_x} \Big( \sum_{u\in \mathbb{G}_{n-q}}\E(Z_q^{(u)}(V)\log^\star Z_q^{(u)}(V)\vert \mathcal{F}_{n-q})   \Big)
   &\leq \lambda^{n-q}\Big(\gamma(i_q)h(x)+V^\star(x)a_{n-q}\Big)\\
&\leq \lambda^{n-q}V^\star(x)\Big(\gamma(V^\star)I_q+a_{n-q}\Big),
\end{align*}
where we used $h\in \mathcal{B}(V^\star)$ and $\gamma(i_q)\leq I_q \gamma(V^\star)$.
Combining these upper-bounds, we finally get 
\begin{align*}
    i_n(x)\leq K_1^q   i_{n-q}(x)+K_1^{n}q\log(K_1)V(x)+\lambda^{n-q}V^\star(x)(\gamma(V^\star)I_q+a_{n-q}).
\end{align*}
Iterating this inequality we obtain
\begin{align*}
      i_n(x)&\leq K_1^{mq}i_{n-mq}(x)+ V(x)mq\log(K_1)K_1^n +V^\star(x) \sum_{k=1}^m \lambda^{n-kq}K_1^{q(k-1)}(\gamma(V^\star)I_q+
    a_{n-kq})\\
    & \leq K_1^{n}i_{r}(x)+ nK_1^{n}V(x)\log(K_1) +cK_2^{n} \, \Big( I_q+ 1\Big)V^\star(x),
\end{align*}
    where  $K_2\geq K_1$ and $c$ are finite  constants, recalling that $\sum a_n/n<\infty$.
     Dividing by $V^\star$, taking the sup over $x\in \mathcal{X}$ and recalling that $V^{\star}\geq V$ yields the result.

\end{proof}

\begin{lemma}
 \label{Lemma-propag-ii} Assume  there exists $C\geq 0$ and that $V$ satisfies  $S_nV\leq C \lambda^n V$ for any $n\geq 0$ and some constant $C\geq 0$. Then there exists a finite constant $K\geq 0$ such that for any $n=mq+r$  with $n\in \mathbb N^*$, $m \in \N^*, q \in \N$ and  $0\leq r \leq m-1$, 
$$I_n \leq n \lambda^n K^m (I_r+I_q+1).$$
\end{lemma}

\begin{proof}
Similarly as in the previous proof, 
\begin{align*}
    \E_{\delta_x}(Z_{n-q}(S_qV)\log^\star Z_{n-q}(S_qV))
&\leq C \lambda^{q} \left(i_{n-q}(x)+q\lambda^{n-q}V(x))\right),
\end{align*}
where $C\geq 0$ is a constant.  
We next continue with the second term. 
\begin{align*}
    \E_{\delta_x}\left( \sum_{u\in \mathbb{G}_{n-q}} \E(Z_q^{(u)}(V)\log^\star Z_q^{(u)}(V)\vert \mathcal{F}_{n-q}) \right)&=\E_{\delta_x}\left( \sum_{u\in \mathbb{G}_{n-q}} \E_{Z^{(u)}}(Z_q(V)\log^\star Z_q(V)) \right)\\
    & =(S_{n-q} i_q)(x)\\
    &\leq \lambda^{n-q}V^\star(x)(I_q\gamma(V^\star)+a_{n-q}).
\end{align*}
Combining these two upper-bounds, we obtain
\begin{align*}
    i_n(x)&\leq C \lambda^q i_{n-q}(x)+Cq\lambda ^n V(x)+\lambda^{n-q}V^\star(x)(I_q\gamma(V^\star)+a_{n-q}),
\end{align*}
which becomes by iteration, relying on   {$\sum_n a_n/n <\infty$} \ $\lambda>1$, 
\begin{align*}
    i_n(x)&\leq \lambda^{mq}C^m  i_{n-mq}(x)+C'qm \lambda^n V(x)+V^\star(x) \sum_{k=1}^m C^{k-1}\lambda^{n-kq}\lambda^{q(k-1)}(I_q+
    a_{n-kq})\\
    &\leq C''\lambda^n \left(C^m  i_{r}(x)+q V(x)+ nC^mC'(I_q+
    1)V^\star(x)\right),
\end{align*}
where $C',C''$ are constants. Dividing by $V^\star$ and taking the sup on each side ends the proof.
\end{proof}

\end{appendix}

\noindent {\bf Acknowledgement.} The authors are very grateful to the anonymous referees for their very thorough  reviews and their numerous relevant suggestions on the first version of this work. This work  was partially funded by the Chair “Mod\'elisation Math\'ematique et Biodiversit\'e" of VEOLIA-Ecole Polytechnique-MNHN-F.X., by the European Union (ERC, SINGER, 101054787), by the Fondation Mathématique Jacques Hadamard and by the ANR project NOLO (ANR-
20-CE40- 0015), funded by the French Ministry of Research. Views and opinions expressed are however those of the author(s) only and do not necessarily reflect those of the European Union or the European Research Council. Neither the European Union nor the granting authority can be held responsible for them. 
\vspace{1mm}
\bibliographystyle{apalike}
\bibliography{sample.bib}

\begin{thebibliography}{}

\bibitem[{Andr{\'e}} and {Duchamps}, 2025]{2025arXiv250305575A}
{Andr{\'e}}, M. and {Duchamps}, J.-J. (2025).
\newblock {Sharp $L \log L$ condition for supercritical Galton-Watson processes
  with countable types}.
\newblock {\em arXiv e-prints}, page arXiv:2503.05575.

\bibitem[Asmussen and Hering, 1976]{asmussen1976strong}
Asmussen, S. and Hering, H. (1976).
\newblock Strong limit theorems for general supercritical branching processes
  with applications to branching diffusions.
\newblock {\em Zeitschrift f{\"u}r Wahrscheinlichkeitstheorie und Verwandte
  Gebiete}, 36(3):195--212.

\bibitem[Athreya, 2000]{athreya2000change}
Athreya, K.~B. (2000).
\newblock Change of measures for markov chains and the llogl theorem for
  branching processes.
\newblock {\em Bernoulli}, pages 323--338.

\bibitem[Bansaye et~al., 2022]{bansaye2022non}
Bansaye, V., Cloez, B., Gabriel, P., and Marguet, A. (2022).
\newblock A non-conservative harris ergodic theorem.
\newblock {\em Journal of the London Mathematical Society}, 106(3):2459--2510.

\bibitem[Bansaye et~al., 2011]{bansaye2011limit}
Bansaye, V., Delmas, J.-F., Marsalle, L., and Tran, V.~C. (2011).
\newblock Limit theorems for markov processes indexed by continuous time
  galton--watson trees.
\newblock {\em The Annals of Applied Probability}, 21(6):2263--2314.

\bibitem[Bansaye et~al., 2023]{bansaye2023growth}
Bansaye, V., Gu, C., and Yuan, L. (2023).
\newblock A growth-fragmentation-isolation process on random recursive trees
  and contact tracing.
\newblock {\em The Annals of Applied Probability}, 33(6B):5233--5278.

\bibitem[Bertoin, 2017]{bertoin2017markovian}
Bertoin, J. (2017).
\newblock Markovian growth-fragmentation processes.

\bibitem[Bertoin and Watson, 2020]{bertoin2020strong}
Bertoin, J. and Watson, A.~R. (2020).
\newblock The strong malthusian behavior of growth-fragmentation processes.
\newblock {\em Annales Henri Lebesgue}, 3:795--823.

\bibitem[Billingsley, 2013]{billingsley2013convergence}
Billingsley, P. (2013).
\newblock {\em Convergence of probability measures}.
\newblock John Wiley \& Sons.

\bibitem[Braunsteins et~al., 2019]{braunsteins2019pathwise}
Braunsteins, P., Decrouez, G., and Hautphenne, S. (2019).
\newblock A pathwise approach to the extinction of branching processes with
  countably many types.
\newblock {\em Stochastic Processes and their Applications}, 129(3):713--739.

\bibitem[Ca{\~n}izo and Mischler, 2023]{canizo2023harris}
Ca{\~n}izo, J.~A. and Mischler, S. (2023).
\newblock Harris-type results on geometric and subgeometric convergence to
  equilibrium for stochastic semigroups.
\newblock {\em Journal of Functional Analysis}, 284(7):109830.

\bibitem[Champagnat and Villemonais, 2016]{champagnat2016exponential}
Champagnat, N. and Villemonais, D. (2016).
\newblock Exponential convergence to quasi-stationary distribution and
  q-process.
\newblock {\em Probability Theory and Related Fields}, 164(1):243--283.

\bibitem[Champagnat and Villemonais, 2023]{champagnat2023general}
Champagnat, N. and Villemonais, D. (2023).
\newblock General criteria for the study of quasi-stationarity.
\newblock {\em Electronic Journal of Probability}, 28:1--84.

\bibitem[Chen et~al., 2017]{chen2017law}
Chen, Z.-Q., Ren, Y.-X., and Yang, T. (2017).
\newblock Law of large numbers for branching symmetric hunt processes with
  measure-valued branching rates.
\newblock {\em Journal of Theoretical Probability}, 30(3):898--931.

\bibitem[Cloez, 2017]{cloez2017limit}
Cloez, B. (2017).
\newblock Limit theorems for some branching measure-valued processes.
\newblock {\em Advances in Applied Probability}, 49(2):549--580.

\bibitem[Cloez et~al., 2021]{cloez2021long}
Cloez, B., de~Saporta, B., and Roget, T. (2021).
\newblock Long-time behavior and darwinian optimality for an asymmetric
  size-structured branching process.
\newblock {\em Journal of Mathematical Biology}, 83(6):69.

\bibitem[Cloez and Gabriel, 2024]{cloez2024fast}
Cloez, B. and Gabriel, P. (2024).
\newblock Fast, slow convergence, and concentration in the house of cards
  replicator-mutator model.
\newblock {\em Differential and Integral Equations}, 37(7/8):547--584.

\bibitem[Del~Moral, 2004]{del2004feynman}
Del~Moral, P. (2004).
\newblock {\em Feynman-kac formulae}.
\newblock Springer.

\bibitem[Del~Moral et~al., 2023]{del2023stability}
Del~Moral, P., Horton, E., and Jasra, A. (2023).
\newblock On the stability of positive semigroups.
\newblock {\em The Annals of Applied Probability}, 33(6A):4424--4490.

\bibitem[Del~Moral and Miclo, 2002]{del2002stability}
Del~Moral, P. and Miclo, L. (2002).
\newblock On the stability of nonlinear feynman-kac semigroups.
\newblock In {\em Annales de la Facult{\'e} des sciences de Toulouse:
  Math{\'e}matiques}, volume~11, pages 135--175.

\bibitem[Engl{\"a}nder, 2009]{englander2009law}
Engl{\"a}nder, J. (2009).
\newblock Law of large numbers for superdiffusions: the non-ergodic case.
\newblock In {\em Annales de l'IHP Probabilit{\'e}s et statistiques},
  volume~45, pages 1--6.

\bibitem[Englander, 2014]{englander2014spatial}
Englander, J. (2014).
\newblock {\em Spatial branching in random environments and with interaction},
  volume~20.
\newblock World Scientific.

\bibitem[Engl{\"a}nder et~al., 2010]{englander2010strong}
Engl{\"a}nder, J., Harris, S.~C., and Kyprianou, A.~E. (2010).
\newblock Strong law of large numbers for branching diffusions.
\newblock In {\em Annales de l'IHP Probabilit{\'e}s et statistiques},
  volume~46, pages 279--298.

\bibitem[Engl{\"a}nder and Kyprianou, 2004]{englander2004local}
Engl{\"a}nder, J. and Kyprianou, A.~E. (2004).
\newblock Local extinction versus local exponential growth for spatial
  branching processes.
\newblock {\em The Annals of Probability}, 32(1A):78--99.

\bibitem[Ethier and Kurtz, 2009]{ethier2009markov}
Ethier, S.~N. and Kurtz, T.~G. (2009).
\newblock {\em Markov processes: characterization and convergence}.
\newblock John Wiley \& Sons.

\bibitem[Gabriel, 2021]{gabriel:tel-03144625}
Gabriel, P. (2021).
\newblock {\em {Asymptotic analysis of non-local equations arising in
  biology}}.
\newblock Habilitation {\`a} diriger des recherches, {Universit{\'e}
  Paris-Saclay, Universit{\'e} Versailles Saint-Quentin-en-Yvelines}.

\bibitem[Gonzalez et~al., 2022]{gonzalez2022asymptotic}
Gonzalez, I., Horton, E., and Kyprianou, A.~E. (2022).
\newblock Asymptotic moments of spatial branching processes.
\newblock {\em Probability Theory and Related Fields}, 184(3):805--858.

\bibitem[Hairer and Mattingly, 2011]{hairer2011yet}
Hairer, M. and Mattingly, J.~C. (2011).
\newblock Yet another look at harris’ ergodic theorem for markov chains.
\newblock In {\em Seminar on Stochastic Analysis, Random Fields and
  Applications VI: Centro Stefano Franscini, Ascona, May 2008}, pages 109--117.
  Springer.

\bibitem[Harris et~al., 2010]{harris2010strong}
Harris, S.~C., Knobloch, R., and Kyprianou, A.~E. (2010).
\newblock Strong law of large numbers for fragmentation processes.
\newblock In {\em Annales de l'IHP Probabilit{\'e}s et statistiques},
  volume~46, pages 119--134.

\bibitem[Harris et~al., 1963]{harris1963theory}
Harris, T.~E. et~al. (1963).
\newblock {\em The theory of branching processes}, volume~6.
\newblock Springer Berlin.

\bibitem[Horton et~al., 2020]{horton2020stochastic}
Horton, E., Kyprianou, A.~E., and Villemonais, D. (2020).
\newblock Stochastic methods for the neutron transport equation i.
\newblock {\em The Annals of Applied Probability}, 30(6):2573--2612.

\bibitem[Horton and Watson, 2020]{horton2020strong}
Horton, E. and Watson, A.~R. (2020).
\newblock Strong laws of large numbers for a growth-fragmentation process with
  bounded cell sizes.
\newblock {\em arXiv preprint arXiv:2012.03273}.

\bibitem[Jagers, 1975]{zbMATH03555176}
Jagers, P. (1975).
\newblock {\em Branching processes with biological applications}.
\newblock Wiley Ser. Probab. Math. Stat. John Wiley \& Sons, Hoboken, NJ.

\bibitem[Jonckheere and Saglietti, 2020]{zbMATH07199305}
Jonckheere, M. and Saglietti, S. (2020).
\newblock On laws of large numbers in {{\(L^2\)}} for supercritical branching
  markov processes beyond {{\(\lambda \)}}-positivity.
\newblock {\em Ann. Inst. Henri Poincar{\'e}, Probab. Stat.}, 56(1):265--295.

\bibitem[Kesten and Stigum, 1966]{kesten1966limit}
Kesten, H. and Stigum, B.~P. (1966).
\newblock A limit theorem for multidimensional galton-watson processes.
\newblock {\em The Annals of Mathematical Statistics}, 37(5):1211--1223.

\bibitem[Kontoyiannis and Meyn, 2003]{kontoyiannis2003spectral}
Kontoyiannis, I. and Meyn, S.~P. (2003).
\newblock Spectral theory and limit theorems for geometrically ergodic markov
  processes.
\newblock {\em The Annals of Applied Probability}, 13(1):304--362.

\bibitem[Kontoyiannis and Meyn, 2012]{kontoyiannis2012geometric}
Kontoyiannis, I. and Meyn, S.~P. (2012).
\newblock Geometric ergodicity and the spectral gap of non-reversible markov
  chains.
\newblock {\em Probability Theory and Related Fields}, 154(1):327--339.

\bibitem[Kurtz et~al., 1997]{kurtz1997conceptual}
Kurtz, T., Lyons, R., Pemantle, R., and Peres, Y. (1997).
\newblock A conceptual proof of the kesten-stigum theorem for multi-type
  branching processes.
\newblock {\em Classical and modern branching processes}, pages 181--185.

\bibitem[Liu et~al., 2011]{liu2011log}
Liu, R.-L., Ren, Y.-X., and Song, R. (2011).
\newblock {L} log {L} condition for supercritical branching hunt processes.
\newblock {\em Journal of Theoretical Probability}, 24(1):170--193.

\bibitem[Liu et~al., 2013]{liu2013strong}
Liu, R.-L., Ren, Y.-X., and Song, R. (2013).
\newblock Strong law of large numbers for a class of superdiffusions.
\newblock {\em Acta applicandae mathematicae}, 123(1):73--97.

\bibitem[Louidor and Saglietti, 2020]{louidor2020strong}
Louidor, O. and Saglietti, S. (2020).
\newblock A strong law of large numbers for super-critical branching brownian
  motion with absorption.
\newblock {\em Journal of Statistical Physics}, 181(4):1112--1137.

\bibitem[Lyons et~al., 1995]{lyons1995conceptual}
Lyons, R., Pemantle, R., and Peres, Y. (1995).
\newblock Conceptual proofs of l log l criteria for mean behavior of branching
  processes.
\newblock {\em The Annals of Probability}, pages 1125--1138.

\bibitem[Marguet, 2019]{marguet19}
Marguet, A. (2019).
\newblock {Uniform sampling in a structured branching population}.
\newblock {\em Bernoulli}, 25(4A):2649 -- 2695.

\bibitem[Meyn and Tweedie, 1993]{meyn1993stability}
Meyn, S.~P. and Tweedie, R.~L. (1993).
\newblock Stability of markovian processes iii: Foster--lyapunov criteria for
  continuous-time processes.
\newblock {\em Advances in Applied Probability}, 25(3):518--548.

\bibitem[Mischler and Scher, 2016]{mischler2016spectral}
Mischler, S. and Scher, J. (2016).
\newblock Spectral analysis of semigroups and growth-fragmentation equations.
\newblock In {\em Annales de l'IHP Analyse non lin{\'e}aire}, volume~33, pages
  849--898.

\bibitem[Moy, 1967]{Moy1967}
Moy, S.-T.~C. (1967).
\newblock Extensions of a limit theorem of everett, ulam and harris on
  multitype branching processes to a branching process with countably many
  types.
\newblock {\em The Annals of Mathematical Statistics}, pages 992--999.

\bibitem[Nummelin, 1984]{Nummelin}
Nummelin, E. (1984).
\newblock {\em General Irreducible Markov Chains and Non-Negative Operators}.
\newblock Cambridge Tracts in Mathematics. Cambridge University Press.

\bibitem[Roelly and Rouault, 1990]{roelly1990construction}
Roelly, S. and Rouault, A. (1990).
\newblock Construction et propri{\'e}t{\'e}s de martingales des branchements
  spatiaux interactifs.
\newblock {\em International Statistical Review/Revue Internationale de
  Statistique}, pages 173--189.

\bibitem[Sharpe, 1988]{sharpe1988general}
Sharpe, M. (1988).
\newblock {\em General theory of Markov processes}, volume 133.
\newblock Academic press.

\bibitem[Toma{\v{s}}evi{\'c} et~al., 2022]{tomavsevic2022ergodic}
Toma{\v{s}}evi{\'c}, M., Bansaye, V., and V{\'e}ber, A. (2022).
\newblock Ergodic behaviour of a multi-type growth-fragmentation process
  modelling the mycelial network of a filamentous fungus.
\newblock {\em ESAIM: Probability and Statistics}, 26:397--435.

\bibitem[Velleret, 2023]{velleret2023exponential}
Velleret, A. (2023).
\newblock Exponential quasi-ergodicity for processes with discontinuous
  trajectories.
\newblock {\em ESAIM: Probability and Statistics}, 27:867--912.

\bibitem[Villemonais and Watson, 2025]{villemonais2025quasi}
Villemonais, D. and Watson, A.~R. (2025).
\newblock A quasi-stationary approach to the long-term asymptotics of the
  growth-fragmentation equation.
\newblock {\em The Annals of Applied Probability}, 35(2):1233--1297.

\bibitem[Williams, 1991]{williams1991probability}
Williams, D. (1991).
\newblock {\em Probability with martingales}.
\newblock Cambridge university press.

\bibitem[Étienne Bernard et~al., 2019]{cyclic}
Étienne Bernard, Doumic, M., and Gabriel, P. (2019).
\newblock Cyclic asymptotic behaviour of a population reproducing by fission
  into two equal parts.

\end{thebibliography}

\end{document}